\documentclass[12pt]{amsart}
\usepackage{amsfonts,amsmath,amssymb,amsthm,latexsym,verbatim}
\input xypic
\newtheorem{theorem}{Theorem}[section]
\newtheorem{lemma}[theorem]{Lemma}

\newtheorem{corollary}[theorem]{Corollary}

\theoremstyle{remark}
\newtheorem{remark}[theorem]{Remark}

\DeclareMathOperator{\image}{{\mathrm{Im}}}
\DeclareMathOperator{\soc}{{\mathrm{soc}}}
\DeclareMathOperator{\head}{{\mathrm{head}}}
\DeclareMathOperator{\rad}{{\mathrm{rad}}}
\newcommand{\id}{{\mathrm{id}}}
\DeclareMathOperator{\Ext}{{\mathrm{Ext}}}
\DeclareMathOperator{\Hom}{{\mathrm{Hom}}}
\DeclareMathOperator{\End}{{\mathrm{End}}}
\newcommand{\abs}[1]{|#1|}	%absolute value
\newcommand{\dprod}[2]{\langle {#1} , {#2}\rangle}	%dot product
\newcommand{\F}{{\mathbf {F}}} 
\newcommand{\R}{{\mathbf {R}}} 
\newcommand{\Z}{{\mathbf {Z}}}

\DeclareMathOperator{\Orth}{{\mathrm {O}}} 	
\DeclareMathOperator{\PO}{{\mathrm {PO}}} 	
\DeclareMathOperator{\Om}{{\mathrm {\Omega}}} 	
\DeclareMathOperator{\POm}{{\mathrm {P\Omega}}} 	
\DeclareMathOperator{\PGU}{{\mathrm {PGU}}} 	
\DeclareMathOperator{\SU}{{\mathrm {SU}}} 	
\DeclareMathOperator{\PSU}{{\mathrm {PSU}}} 	
\DeclareMathOperator{\rank}{{\mathrm {rank}}} 	%rank
 	
\DeclareMathOperator{\Ch}{{\mathrm {Ch}}} 	
\DeclareMathOperator{\sign}{{\mathrm {sign}}} 	
\DeclareMathOperator{\ind}{{\mathrm {ind}}} 	%induced

\newcommand{\allone}{{\mathbf 1}} 	%all-one vector
\newcommand{\ds}{\displaystyle}
\newcommand{\dt}{\langle \lambda+\rho, \alpha \rangle}
\newcommand{\n}{\lambda+\rho - m\,p\,\alpha}

\newcommand{\el}{\frac{\ell}{2}}
\newcommand{\bir}{\lambda+\rho - m\,p\,\alpha}
\newcommand{\kay}{\chi(\lambda - m\,p\,\alpha)}

\begin{document}
\title[Some simple modules for classical groups]{Some simple modules for classical groups and $p$-ranks of orthogonal and Hermitian geometries}
\author{Ogul Arslan}
\address{Department of Mathematics and Statistics\\Coastal Carolina University\\ P. O. Box 261954\\ Conway SC 29528-6054\\ USA}
\author{Peter Sin}
\address{Department of Mathematics\\University of Florida\\ P. O. Box 118105\\ Gainesville
FL 32611\\ USA}
\date{August 4th, 2009}
\begin{abstract}
We determine the characters of the simple composition 
factors and the submodule lattices of certain Weyl modules for classical groups.
The results have several applications. The simple modules
arise in the study of incidence systems in finite geometries and 
knowledge of their dimensions yields the $p$-ranks of these incidence systems.
\end{abstract}
\maketitle

%}}}

%{{{ Introduction

\section{Introduction} 
In this paper we  study some special cases of the general
problem of determining the formal characters of irreducible rational
representations of a semisimple algebraic group $G$ over an
algebraically closed field $k$ of characteristic $p>0$.
The groups we shall consider are classical groups and
the simple modules are related to their standard modules.
To be more precise, if we number the fundamental weights 
so that the first one, $\omega_1$  is 
the highest weight of the standard module, then simple
modules to be studied are those having highest weight
$\lambda$ in the following list.

\begin{enumerate}
\item[(B)] For $G$ of type $B_\ell$, ($\ell\geq 2$)  $\lambda=r(\omega_1)$, $0\leq r\leq p-1$;

\item[(D)]For $G$ of type $D_\ell$, ($\ell\geq 3$) $\lambda=r(\omega_1)$, $0\leq r\leq p-1$;

\item[(A)] For $G$ of type $A_\ell$, ($\ell\geq 3$)  $\lambda=r(\omega_1+\omega_\ell)$, $0\leq r\leq p-1$;

\item[(A$^\prime$)]For $G$ of type $A_\ell$, ($\ell\geq 4$)  $\lambda=\omega_2+\omega_{\ell-1}$; and

\item[(A$^{\prime\prime}$)] For $G$ of type $A_4$, $\lambda=(p-2)(\omega_2+\omega_{\ell-1})$
or $(p-1)(\omega_2+\omega_{\ell-1})$.
\end{enumerate}

For each weight in this list, our purpose is to give a complete
description of the composition factors and the submodule structure of the 
corresponding Weyl module $V(\lambda)$ or, dually, the induced module $H^0(\lambda)$.
The formal characters of these modules are given by Weyl's Character formula.

We should also mention some weights which do not appear in above list 
but which play the same role in applications. These are the weights
$\lambda=r\omega_1$, $0\leq r\leq p-1$ for groups of type $A_\ell$ and $C_\ell$.
The Weyl modules for these weights are isomorphic to symmetric
powers of the standard module and are all simple.

\subsection{Statements of results}
The main tool we use for examining the deeper structure of Weyl modules is
the Jantzen Sum Formula (\cite[II.8.19]{Ja}). 
In general, due to  the recursive nature of its application, 
the formula can be expected to yield only partial information about 
composition factors. In certain cases one can obtain precise results, as demonstrated in the original paper \cite{Ja2}.
The weights in the list above give further examples 
where complete results can be obtained.
For these cases, the combinatorics of the Sum Formula can be kept under control
and we obtain accurate bounds which, when supplemented with certain properties of  good filtrations  and classical facts about tensor products, yield the full submodule 
structure of the Weyl modules.

We can now state our main results.

\begin{theorem}\label{ThmB} Let $G$ be of type $B_\ell$, $\ell\geq 2$. Let
$\omega_1$ be the highest weight of the standard orthogonal module of dimension
$2\ell+1$. Assume $0\leq r\leq p-1$. 
Then the following hold.
\begin{enumerate}
\item[(a)] $H^0(r\omega_1)$ is simple unless
(i) $p=2$ and $r=1$ or (ii) $p>2$ and there exists a positive \emph{odd} integer $m$ such that
\begin{equation*}
%(m-1)p < 2\ell-3 <mp.
r+2\ell-1\leq mp\leq 2r+2\ell-2.
\end{equation*}  
\item[(b)] If (i) holds  then the  quotient $H^0(\omega_1)/L(\omega_1)$ 
is the one-dimensional trivial module.
\item[(c)] If (ii) holds then  $m$ is unique and 
\begin{equation*} 
H^0(r\omega_1)/L(r\omega_1)\cong H^0(r_1\omega_1),
\end{equation*}
where $r_1=mp-2\ell+1-r$.
Furthermore the module $H^0(r_1\omega_1)$ is simple.
\item[(d)] We have
\begin{equation*}
\dim L(r\omega_1)=\begin{cases}
2\ell,\qquad\text{if  (i) holds}.\\
\binom{2\ell+r}{2\ell}-\binom{2\ell+r-2}{2\ell}-\binom{mp-r+1}{2\ell}+\binom{mp-r-1}{2\ell},\qquad\text{if (ii) holds.}\\
\binom{2\ell+r}{2\ell}-\binom{2\ell+r-2}{2\ell}, \qquad\text{otherwise.}
\end{cases}
\end{equation*}
\end{enumerate}
\end{theorem}

\begin{theorem}\label{ThmD} Let $G$ be of type $D_\ell$, $\ell\geq 3$.
 Let $\omega_1$ be the highest weight of the standard orthogonal module of dimension $2\ell$. Assume $0\leq r\leq p-1$.  Then the following hold.
\begin{enumerate}
\item[(a)] Suppose that
 there exists a positive \emph{even} integer $m$ such that
\begin{equation*}
r+2\ell-2 \leq mp\leq 2r+2\ell-3.
\end{equation*}
Then $m$ is unique and 
\begin{equation*}
H^0(r\omega_1)/L(r\omega_1)\cong H^0(r_1\omega_1),
\end{equation*}
where 
$r_1=mp-2\ell+2-r$. 
Furthermore the module $H^0(r_1\omega_1)$ is simple.
\item[(b)] Otherwise, $H^0(r\omega_1)$ is simple.
\item[(c)] We have
\begin{equation*}
\dim L(r\omega_1)=\begin{cases}
\binom{2\ell+r-1}{2\ell-1}-\binom{2\ell+r-3}{2\ell-1}-\binom{mp-r+1}{2\ell-1}+\binom{mp-r-1}{2\ell-1}   ,\qquad\text{in (a).}\\
\binom{2\ell+r-1}{2\ell-1}-\binom{2\ell+r-3}{2\ell-1}, \qquad\text{in (b).}
\end{cases}
\end{equation*}
\end{enumerate}
\end{theorem}

For groups of type $A_\ell$ we number the  fundamental weights be so that 
$\omega_i$ is the highest weight of the $i$-th exterior power
of the $(\ell+1)$-dimensional standard module.

\begin{theorem}\label{ThmA} Let $G$ be of type $A_\ell$, $\ell\geq 3$.
Assume $0\leq r\leq p-1$. Then the following hold.
\begin{enumerate}
\item[(a)] Suppose that there exists a positive integer $m$ such that
\begin{equation*}
r+\ell\leq mp\leq 2r+\ell-1.
\end{equation*}
Then $m$ is unique and 
\begin{equation*}
H^0(r(\omega_1+\omega_\ell))/L(r(\omega_1+\omega_\ell))\cong 
H^0(r_1(\omega_1+\omega_\ell)),
\end{equation*}
where $r_1=mp-\ell -r$. 
Furthermore the module $H^0(r_1(\omega_1+\omega_\ell))$ is simple.
\item[(b)] Otherwise, $H^0(r(\omega_1+\omega_\ell))$ is simple.
\item[(c)] We have
\begin{equation*}
\dim L(r(\omega_1+\omega_\ell))=\begin{cases}
\binom{\ell+r}{\ell}^2-\binom{\ell+r-1}{\ell}^2 -\binom{mp-r}{\ell}^2+\binom{mp-r-1}{\ell}^2, \qquad\text{in (a).}\\
\binom{\ell+r}{\ell}^2-\binom{\ell+r-1}{\ell}^2, \qquad\text{in (b).}
\end{cases}
\end{equation*}
\end{enumerate}
\end{theorem}

\begin{theorem}\label{ThmA'} Let $G$ be of type $A_\ell$, $\ell\geq 4$.
If $p>2$  then the following hold.
\begin{enumerate}
\item[(a)] If $\ell = 0\pmod p$ then
$H^0(\omega_2 + \omega_{\ell-1})/L(\omega_2 + \omega_{\ell-1})\cong k$.
\item[(b)] If  $\ell = 1 \pmod p$ then
\begin{equation*}
H^0(\omega_2 + \omega_{\ell-1})/L(\omega_2 + \omega_{\ell-1})\cong 
H^0(\omega_1+\omega_{\ell})
\end{equation*}
and this module is simple.
\item[(c)] In all other cases $H^0(\omega_2+\omega_{\ell-1})$ is simple.
\end{enumerate}

If $p=2$ then the following hold.
\begin{enumerate}
\item[(d)] If $\ell\equiv 0\pmod 4$ then 
\begin{equation*}
H^0(\omega_2 + \omega_{\ell-1})/L(\omega_2 + \omega_{\ell-1})\cong k.
\end{equation*}
\item[(e)] If $\ell\equiv 1\pmod 4$ then
\begin{equation*}
H^0(\omega_2 + \omega_{\ell-1})/L(\omega_2 + \omega_{\ell-1})\cong 
V(\omega_1+\omega_\ell).
\end{equation*}
\item[(f)] If  $\ell\equiv 2\pmod 4$ then $H^0(\omega_2 + \omega_{\ell-1})$ 
is simple. 
\item[(g)] If  $\ell\equiv 3\pmod 4$ then
\begin{equation*}
H^0(\omega_2 + \omega_{\ell-1})/L(\omega_2 + \omega_{\ell-1})\cong L(\omega_1+\omega_{\ell}).
\end{equation*}
\end{enumerate}

We have
\begin{enumerate}
\item[(h)]
\begin{equation*}
\dim L(\omega_2+\omega_{\ell-1})=\begin{cases}
    \frac{1}{4}\ell^4+\frac{1}{2}\ell^3-\frac{3}{4}\ell^2 -2\ell-2, \qquad \text{in (a) and (d).} \\
    \frac{1}{4}(\ell+2)(\ell^3-7\ell-2), \qquad \text{in  (b) and (e).}\\
    \frac{1}{4}(\ell-2)(\ell+2)(\ell+1)^2, \qquad\text{in (c) or (f).}\\
    \frac{1}{4}\ell(\ell+3)(\ell^3+2\ell^2-7\ell-16), \qquad\text{in (g).}
\end{cases}
\end{equation*}
\end{enumerate}

\end{theorem}

\begin{theorem} \label{ThmA''} Let $G$ be of type $A_4$.
\begin{enumerate}
\item[(a)] If $p=2$, then $H^0((p-1)(\omega_2+\omega_3))/L((p-1)(\omega_2+\omega_3))\cong k$.
\end{enumerate}

If $p>2$, then the following hold.
\begin{enumerate}
\item[(b)] $H^0((p-1)(\omega_2+\omega_3))/L((p-1)(\omega_2+\omega_3))\cong
L((p-2)(\omega_2+\omega_3))$.
\item[(c)] $H^0((p-2)(\omega_2+\omega_3))/L((p-2)(\omega_2+\omega_3))\cong
H^0((p-2)(\omega_1+\omega_4))$, which is simple.
\item[(d)] If $p>2$, then
\begin{multline*}
\dim L((p-1)(\omega_2+\omega_3))=\\
\frac{p(p+1)}{32}\binom{2p+2}{3}^2-\frac{p(p-1)}{32}\binom{2p}{3}^2+\frac{p}{2}\binom{p+1}{3}^2.
\end{multline*}
\item[(e)] If $p=2$, then $\dim L(\omega_2+\omega_3)=74$.
\end{enumerate}
\end{theorem}

\subsection{Some applications}\label{prankresults}

One motive for studying the simple modules in cases (B),
(D) and (A) above comes from problems about incidence matrices arising from Hermitian and orthogonal geometries. Let $q=p^t$ and let $V(q)$  be a vector space of dimension $n$ over the field $\F_q$ of $q$ elements. We assume that
$V(q)$ carries a nondegenerate quadratic form.
Let $\widehat P$ be the set of all one-dimensional subspaces
and $P$ be the set of  singular one-dimensional subspaces.
Let ${\widehat P}^*$ be the set of all $(n-1)$-dimensional subspaces in $V$
and $P^*$ the set of polar hyperplanes $\langle v\rangle^\perp$ for
$\langle v\rangle\in P$. Then we may consider various incidence
systems such as $(\widehat{P},{\widehat P}^*)$, $(P,{\widehat P}^*)$ 
and $(P,P^*)$ under the  incidence relation of inclusion.

Similarly, if $V(q^2)$ is a vector space of dimension $n$ over $\F_{q^2}$ with
a nondegenerate Hermitian form then we can  consider the incidence
relations of inclusion involving the set $\widehat P$ of all
one-dimensional spaces, the set $P$ of
singular one-dimensional subspaces, the set ${\widehat P}^*$ of
all hyperplanes and the set $P^*$ of polar hyperplanes of 
singular one-dimensional subspaces.

These incidence relations were studied by
Blokhuis and Moorhouse \cite{BM} and Moorhouse \cite{Mo}.
We can order the sets $\widehat{P}$ and ${\widehat P}^*$ so that
the incidence matrix $A_1$ of $(P,{\widehat P}^*)$ and
the incidence matrix $A_{11}$ of  $(P,P^*)$
are submatrices of the incidence matrix $A$ of 
$(\widehat{P},{\widehat P}^*)$ in the following way.
\begin{equation}
A=\begin{pmatrix}A_{11}&A_{12}\\A_{21}&A_{22}\end{pmatrix}, \qquad A_1=\begin{pmatrix}A_{11}&A_{12}\end{pmatrix}
\end{equation}

We may consider these $0$-$1$ matrices over any commutative ring with 1
and try to compute their invariants. One such invariant is
the $p$-rank, the rank of the matrix regarded as having entries in 
the field of $p$ elements. The $p$-rank of $A$ is well known from
 \cite{MM}, \cite{Sm} and \cite{GD}. 
In \cite{BM} and \cite{Mo} the $p$-ranks of  the $A_1$ were determined,
while the problem of finding the $p$-ranks of the $A_{11}$ was left open.
%%Changed a little and added sentences to relate n and \ell

These $p$-ranks may be deduced from our results on simple modules;
the results for orthogonal spaces of dimension $n$ 
are consequences of our results for types $B_\ell$ with $n=2\ell+1$, 
and types $D_\ell$ with $n=2\ell$, while 
the $p$-ranks for $n$-dimensional Hermitian spaces come from our results
for type $A_\ell$ with $n=\ell+1$.

\begin{theorem}\label{Oprank} Suppose we are in the  orthogonal case with 
$\dim V(q)=n\geq 4$.
The following hold.
\begin{enumerate}
\item[(a)] Assume $p=2$. Then 
\begin{equation*}
\rank_pA_{11}=\begin{cases} 1+n^t, \quad\text{if $n$ is even,}\\
                            1+(n-1)^t,\quad\text{if $n$ is odd.}
\end{cases}
\end{equation*}
\item[(b)] Assume $p>2$. Then the $p$-rank depends on whether
there exists a positive integer $u$ such that 
\begin{equation*}
u\equiv n\pmod 2\quad\text{and}\quad n-3 \leq up\leq p+n-5.
\end{equation*}
If $u$ exists then
\begin{equation*}
\rank_pA_{11}=
1+\left[\binom{n+p-2}{n-1}-\binom{n+p-4}{n-1}-\binom{up+2}{n-1}+\binom{up}{n-1}\right]^t.
\end{equation*}
Otherwise,
\begin{equation*}
\rank_pA_{11}=
1+ \left[\binom{n+p-2}{n-1}-\binom{n+p-4}{n-1} \right]^t.
\end{equation*}
\end{enumerate}
\end{theorem}

\begin{remark} When $n$ is even, there are two types of nondegenerate
forms, distinguished by the Witt index. However,
the $p$-rank of $A_{11}$ is the same for both types. The same was
also true for $A_1$, as shown in \cite{BM}.
\end{remark}

Our proof of Theorem~\ref{Oprank} will not include the case $n=4$, but 
the $p$-rank is well known in that case and fits our formula.

%\begin{theorem}\label{Bprank} In the orthogonal case with $n$ odd, $n\geq5$,  we have
%There are two cases for the rank depending on the existence of a positive integer $m$ satifying
%\begin{equation}\label{m18}
%n-3 \leq up\leq p+n-5
%\end{equation}
%\begin{equation}
%\rank_pA_{11}=\begin{cases} 1+(n-1)^t, \qquad\text{if p=2 and r=1}\\
%1+\left[\binom{n+p-2}{n-1}-\binom{n+p-4}{n-1}-\binom{up+2}{n-1}+\binom{up}{n-1}\right]^t, \quad\text{ if (\ref{m18})
%holds} \\
%1+ \left[\binom{n+p-2}{n-1}-\binom{n+p-4}{n-1}\right]^t, \quad\text{ otherwise}
%\end{cases}
%\end{equation}
%\end{theorem}

\begin{theorem}\label{Aprank} Suppose we are in the Hermitian case with
$\dim V(q^2)=n\geq 4$. There are two cases for the rank depending on the existence of a positive integer $u$ satisfying
\begin{equation*}\label{m19}
n-2 \leq up\leq p+n-4
\end{equation*}
If $u$ exists then
\begin{equation*}
\rank_pA_{11}=
    1+\left[\binom{n+p-2}{n-1}^2-\binom{n+p-3}{n-1}^2 -\binom{up+1}{n-1}^2+\binom{up}{n-1}^2\right]^t.
\end{equation*}
Otherwise,
\begin{equation*}
\rank_pA_{11}=
1+\left[\binom{n+p-2}{n-1}^2-\binom{n+p-3}{n-1}^2\right]^t.
\end{equation*}
\end{theorem}

%changed, put in terms of n
When $n=4$ or $5$, the totally isotropic subspaces of dimensions one and 
two form the points and lines of the Hermitian generalized quadrangle.
The $p$-rank of the incidence relation of points and lines of this 
generalized quadrangle is still unknown in general. Each generalized
quadrangle has a {\it dual}, obtained by interchanging the
roles of the points and lines, while keeping the same incidence relation.
By a polar hyperplane in the dual quadrangle, we mean
the set of dual points which are collinear (in the dual sense)
with a given dual point. 
Taking duals does not affect the $p$-rank of the point-line incidence
relation. However, the point-hyperplane incidence
relations for a quadrangle and its dual have different $p$-ranks in general.
When $n=4$, the dual Hermitian  generalized quadrangle $DH(3,q^2)$ is isomorphic to
the generalized quadrangle $Q(5,q)$ arising from the totally isotropic subspaces of a $6$-dimensional
quadratic module of Witt index $2$, so the $p$-rank of $A_{11}$ in this case is given by
Theorem~\ref{Oprank}, with $\ell=3$.
 The modules in case (A$^{\prime\prime}$) arise from considering the point-hyperplane incidences of the dual of the Hermitian generalized quadrangle for $n=5$.

\begin{theorem}\label{dualHermitian}
 The $p$-rank of the point-hyperplane incidence matrix $A_{11}$
for the dual Hermitian generalized quadrangle $DH(4,q^2)$ is
as follows.
\begin{enumerate}
\item[(a)] If $p>2$ then
\begin{multline*}
\rank_pA_{11}=
1+ \left[\ds \frac{p(p+1)}{32}\binom{2p+2}{3}^2-\frac{p(p-1)}{32}\binom{2p}{3}^2+\frac{p}{2}\binom{p+1}{3}^2\right]^t.
\end{multline*}
\item[(b)] If $p=2$ then $\rank_2A_{11}= 1+74^t$.
\end{enumerate}
\end{theorem}

The reduction of the $p$-rank problem to simple modules is achieved by
reformulating it in terms of representations of the finite classical
group of automorphisms of the given form. This is explained in
\S\ref{prankproofs}.

The simple modules in case (A$^\prime$) were brought to our attention by
Pham Huu Tiep, who asked about their dimensions as well as those 
in (B), (D) and (A) in connection with work  on heights of characters in 
blocks \cite{Tiep}. A. Kleschchev has informed us that Theorem~\ref{ThmA'}
is included among the results in the Ph.D. thesis (1992, Moscow State University) 
of A. Adamovich, where the Weyl modules with highest weight equal to the sum of two fundamental weights are treated. To our knowledge there is no published proof,
but the statement is reproduced in \cite{Kles}.

%}}}

%{{{ Algebraic groups and their representations

\section{Semisimple algebraic groups}\label{alggps}
We shall use the standard reference
\cite{Ja} for notation and background results 
in our discussion of algebraic groups and their representations.
Let $G$ be a simply connected, semisimple algebraic group over an algebraically
closed field $k$ of characteristic $p>0$. 
Let $T$ be a maximal torus in $G$ and let $\ell$ denote its rank. Let
$(X(T), R, Y(T), R^\vee)$ be the {\it root datum} determined by $G$ and $T$
and let $W$ be its Weyl group.
Here $X(T)=\Hom(T,k^\times)$ is the group of rational characters of $T$,
$R\subset X(T)$ is the set of roots, $Y(T)=\Hom(k^\times, T)$ and 
$R^\vee\subset Y(T)$ is the set of coroots. There is a natural pairing
$X(T)\times Y(T)\to \Z$ such that for $\lambda\in X(T)$ and $\phi\in Y(T)$,
$\langle \lambda,\phi\rangle$ is the integer such that the endomorphism
of $k^\times$ given by $\lambda\circ\phi$ is the map $a\mapsto a^{\langle \lambda,\phi\rangle}$.
Let $R^+$ be a positive system in $R$ and $S=\{\alpha_1,\ldots,\alpha_\ell\}$ the corresponding set of simple roots. Let  $B$ be a Borel subgroup containing 
$T$ which corresponds to the set $-R^+$ of negative 
roots. 
The simply connectedness of $G$ means that $Y(T)=\Z R^\vee$.
Then $X:=X(T)$ has a basis consisting of the  fundamental weights 
$\omega_1$, \dots, $\omega_\ell$, defined by
$\langle \omega_i,\alpha^\vee_j\rangle=\delta_{ij}$ for $1\leq i,j\leq \ell$.
 The weights in the set
$$
X_+=\{\lambda\in X\mid \langle\lambda,\alpha^\vee\rangle\geq 0, \forall \alpha\in R^+\}
$$
are called {\it dominant} weights. They are the nonnegative integer combinations of
the fundamental weights. As usual, the half-sum of the positive roots is denoted by $\rho$.
The real vector space $E=X\otimes_\Z\R$ can be given an inner product $\langle-,-\rangle$
such that the set $R$ becomes a root system in the sense of \cite[2.1]{Ca}.
Then $Y(T)$ can be identified with the sublattice of $E$ which
is dual to $X$ and for each root $\alpha$, its coroot $\alpha^\vee$ is identified with the element $2\alpha/\langle\alpha,\alpha\rangle$.

\section{Weyl modules} 
The set $X_+$ parametrizes the simple $G(k)$-modules, but also
two other sets of modules, the Weyl modules $V(\lambda)$
and the induced modules $H^0(\lambda)$. The induced modules are also called dual
Weyl modules because each $H^0(\lambda)$ is the dual module of
$V(\lambda^*)$, where $\lambda^*=-w_0(\lambda)$
and $w_0$ is the longest element of $W$. In the case of the types $B_\ell$ and $D_\ell$,
we actually have $\lambda^*=\lambda$.
The induced modules $H^0(\lambda)$ have the property that
they have a unique simple submodule, and this simple module
is $L(\lambda)$. The Weyl modules and induced modules  can be defined
for Chevalley groups over any algebraically closed field. 
Their weight multiplicities are independent of the field characteristic
and are given by Weyl's Character Formula.
The modules $V(\lambda)$ and $H^0(\lambda)$ have the same formal characters.
In characteristic 0 they are irreducible.
In the case of the types $B_\ell$ and $D_\ell$, it 
is easy to see that $H^0(\omega_1)$ is the standard orthogonal
module, which we denote by $V$. It is  not hard to check
using Weyl's dimension formula that for $r\geq 1$, $H^0(r\omega_1)$
is the quotient $S^r(V^*)/fS^{r-2}(V^*)$, where
$S^m(V^*)$ is the space of homogeneous polynomials of degree $r$
on $V$  and $f\in S^2(V^*)$ is the quadratic form defining the orthogonal group.

\subsection{Weyl module dimensions}\label{Weyldimensions}
Here we give the dimensions of the Weyl modules for the weights (B), (D),
(A), (A$^\prime$) and (A$^{\prime\prime}$) of the Introduction.
For type $B_\ell$, $\ell\ge 2$, we have
\begin{equation}
\dim H^0(r\omega_1)=\binom{2\ell+r}{2\ell}-\binom{2\ell+r-2}{2\ell}
\end{equation}
and for type $D_\ell$, $\ell\ge 3$, we have
\begin{equation}
\dim H^0(r\omega_1)=\binom{2\ell+r-1}{2\ell-1}-\binom{2\ell+r-3}{2\ell-1}.
\end{equation}
For type $A_\ell$, $\ell\ge 3$, Weyl's formula gives
\begin{equation}
\dim H^0(r(\omega_1+\omega_\ell))=\binom{r+\ell}{\ell}^2-\binom{r+\ell-1}{\ell}^2,
\end{equation}
and for $\ell\geq 4$,
\begin{multline}
\dim H^0(r(\omega_2+\omega_{\ell-1}))=\\
\binom{r+\ell-2}{r+1}^2\binom{r+\ell-3}{r}^2\frac{(2r+\ell)(2r+\ell-1)^2(2r+\ell-2)}{\ell(\ell-1)^2(\ell-2)^3}.
\end{multline}
%include dims  from summary

\section{Jantzen Sum Formula}
The Weyl module $V(\lambda)$ has a descending  filtration, defined by Jantzen,
of submodules $V(\lambda)^i$, $i>0$. These submodules are
a little mysterious; they are eventually all zero, but it can
happen that  nonzero terms can be equal. One good property
is that 
\begin{equation}
V(\lambda)^1=\rad V(\lambda),\quad\text{so}\quad V(\lambda)/V(\lambda)^1
\cong L(\lambda).
\end{equation}
Another good property is that we have a formula for the character
of the (external) direct sum of the $V(\lambda)^i$, the {\it Jantzen Sum Formula}:

\begin{equation}\label{jsum}
\sum_{i>0}\Ch( V(\lambda)^i)= -\sum_{\alpha>0}
\sum_{0<mp<\dprod{\lambda+\rho}{\alpha^\vee}}v_p(mp)\chi(\lambda-mp\alpha)
\end{equation}
The left hand side  is the sum of the characters of the $V(\lambda)^i$.
The sign of a Weyl group element is its determinant as a transformation of $E$.
The number $v_p(n)$ is the exponent of $p$ in a prime power 
factorization of $n$.
Next we explain the meaning of $\chi(\mu)$,
for $\mu\in X$. If $\mu\in X_+$ then $\chi(\mu)$ is the character
of $V(\mu)$ (and $H^0(\mu)$), given by Weyl's Character Formula. 
For general $\mu$, there is 
a unique Weyl conjugate $\mu'$ of $\mu+\rho$ in $X_+$. Then $\mu'-\rho$
may or may not be in $X_+$. If not then $\chi(\mu)=0$. If so, then
$\chi(\mu)=\sign(w)\chi(\mu'-\rho)$, where $w\in W$ satisfies $w(\mu+\rho)=\mu'$.

Since the characters $\chi(\mu)$ for $\mu\in X_+$ are linearly independent,
the Sum Formula will tell us whether $V(\lambda)$ is simple or not.
In the case that is not simple, the formula gives incomplete
information about the composition factors, because the LHS of (\ref{jsum})
{\it overestimates} the actual weight multiplicities. Nevertheless, 
there are certain situations where this information is very useful.

The right hand side can be computed by the following  algorithm: 
For each positive root $\alpha$, 
\begin{enumerate}
\item[(i)] Compute $\dprod{\lambda+\rho}{\alpha^\vee}$
\item[(ii)]Compute $\lambda+\rho-mp\alpha$ for 
 $0<m<\dprod{\lambda+\rho}{\alpha^\vee}$
\item[(iii)] Find the Weyl group  conjugate $w(\lambda+\rho-mp\alpha)$ in 
$X_+$ and note the sign of a Weyl group element $w$.
\item[(iv)] Compute $w(\lambda+\rho-mp\alpha)-\rho$.
\item[(v)] The contribution to the sum (\ref{jsum}) is 
$-\sign(w)v_p(mp)\chi(w(\lambda+\rho-mp\alpha)-\rho)$. This will
be zero if $w(\lambda+\rho-mp\alpha)-\rho$ is not in the dominant
Weyl chamber, and is otherwise given by Weyl's Character Formula.
\end{enumerate}

%}}}

%{{{ The basic shortcut

\section{Weyl groups and root systems of types $B_\ell$, $D_\ell$ and $A_\ell$}
For types $B_\ell$ and $D_\ell$ let $e_i$, $i=1$,\dots, $\ell$ be an orthonormal basis of
$E=\R\otimes_\Z X$. Then the inner product of $E$ becomes the usual dot product.

\subsection{Type $B_\ell$} The root system $R$
can be taken to consist of the vectors $\pm e_i\pm e_j$, $1\leq i< j\leq \ell$, together with the vectors $\pm e_i$, $1\leq i\leq \ell$.
The roots $\pm e_i\pm e_j$ are equal to their coroots, but
for $\alpha=e_i$ we have $\alpha^\vee=2e_i$.
The Weyl group consists of all possible permutations of the indices
and all possible sign changes of the $e_i$, so has order $\ell!2^\ell$.
A fundamental system of roots is
\begin{equation*}
S=\{\alpha_i= e_i-e_{i+1}, (1\leq i\leq \ell-1),\quad \alpha_\ell=e_{\ell}\}
\end{equation*} 
The fundamental weights are
\begin{equation*}
\omega_i=(\underbrace{1,1,\ldots, 1}_i,0,\ldots 0)\ \text{for
$1\leq i\leq \ell-1$,}\quad  \omega_\ell=(\frac12,\frac12,\ldots,\frac12)
\end{equation*}
and we have
\begin{equation*}
\rho=(\ell-\frac12,\ell-\frac32,\ldots,\frac12).
\end{equation*}
Let $\mu=(a_1,\ldots,a_\ell)\in X$. Then
\begin{equation*}
\mu\in X_+ \iff a_1\geq a_2\geq\cdots\geq a_\ell\geq 0.
\end{equation*}

\subsection{Type $D_\ell$} 
Here $R=R^\vee$ may be
taken to consist of the vectors $\pm e_i\pm e_j$, $1\leq i< j\leq \ell$.
Roots and coroots are the same in this case.
The Weyl group consists of all possible permutations of the indices
and sign changes of an even number of the $e_i$, so has order $\ell!2^{\ell-1}$.
A fundamental system of roots is
\begin{equation*}
S=\{\alpha_i= e_i-e_{i+1}, (1\leq i\leq \ell-1),\quad \alpha_\ell=e_{\ell-1}+ e_{\ell}\}
\end{equation*} 
The fundamental weights 
are
\begin{multline*}
\omega_i=(\underbrace{1,1,\ldots, 1}_i,0,\ldots 0)\ \text{for
$1\leq i\leq \ell-2$,}\\
\omega_{\ell-1}=(\frac12,\ldots,\frac12,-\frac12), \ 
\omega_\ell=(\frac12,\ldots,\frac12,\frac12)
\end{multline*}
and we have
\begin{equation*}
 \rho=(\ell-1,\ell-2,\ldots, 1, 0).
\end{equation*}
Let $\mu=(a_1,\ldots,a_\ell)\in X$. Then
\begin{equation*}
\mu\in X_+ \iff \text{$a_1\geq a_2\geq\cdots\geq a_\ell$, $a_i\geq 0$ for $1\leq i\leq \ell-1$, and $a_{\ell-1}+a_\ell \geq 0$.}
\end{equation*}

\subsection{Type $A_\ell$}
We let $e_1$, \dots, $e_{\ell+1}$, be
an orthonormal basis of an $(\ell+1)$ dimensional Euclidean space
and we identify $E$ with the hyperplane perpendicular to $e_1+e_2+\cdots+e_{\ell+1}$.
We may take $R=R^\vee$ to be the set of vectors $e_i-e_j$, $1\leq i,j\leq \ell+1$, $i\neq j$. The Weyl group is the set of all permutations of the $\ell+1$ indices.
A fundamental system is
\begin{equation*}
S=\{\alpha_i= e_i-e_{i+1}, (1\leq i\leq \ell)\}
\end{equation*} 

The fundamental dominant weights  $\omega_i$  for $1\leq i\leq\ell$ 
are
\begin{equation*}
\omega_i=\frac{1}{\ell+1}(\underbrace{\ell-i+1,\ldots, \ell-i+1}_i,-i,\ldots -i)
\end{equation*}
and we have
\begin{equation*}
 \rho=(\frac{\ell}{2},\frac{\ell}{2}-1,\ldots,1-\frac{\ell}{2},-\frac{\ell}{2}).
\end{equation*}
Let $\mu=(a_1,\ldots,a_\ell)\in X$. Then
\begin{equation*}
\mu\in X_+ \iff a_1\geq a_2\geq\cdots\geq a_\ell.
\end{equation*}
%rewrite this part to be correct for all types.
\subsection{A shortcut} Suppose $R$ is of type $B_\ell$ or $D_\ell$.
After computing $\lambda+\rho-mp\alpha$ in step (ii) we may
see that its coordinate vector has two entries with the same absolute
values. Then the same will hold for its conjugate in $X_+$. But then since
the coordinates of $\rho$ are strictly decreasing and nonnegative, we see that
subtracting $\rho$ from this conjugate will take us out of $X_+$.
Similarly, if $R$ is of type $A_\ell$, if $\lambda+\rho-mp\alpha$
has two equal coordinates, then subtracting $\rho$
will result in a weight outside $X_+$.

\begin{lemma}\label{critstar}\ 
\begin{enumerate}
\item[(a)] If $R$ is of type $B_\ell$ or $D_\ell$ and
 $\lambda+\rho-mp\alpha$ has two coordinates  with the same absolute value
then  the pair $(\alpha, m)$ contributes nothing to the final sum.
\item[(b)] If $R$ is of type $A_\ell$
and  $\lambda+\rho-mp\alpha$ has two equal coordinates,
then  the pair $(\alpha, m)$ contributes nothing to the final sum
(\ref{jsum}).
\end{enumerate}
\end{lemma}
\qed

%}}}

%{{{  B computations

\section {Computations for $B_\ell$}
\subsection{Evaluation of the Sum Formula}
We fix $r$ with $0\leq r\leq p-1$ and set $\lambda=r\omega_1=(r,0,\ldots,0)$.
Here, we determine the contribution of each positive root to the right hand side
of (\ref{jsum}).
We have $\rho=(\ell-\frac12,\ell-\frac32,\ldots,\frac12)$, so
\begin{equation*} 
\lambda+\rho=(r+\ell-\frac12,\ell-\frac32,\ldots,\underbrace{\ell-i+\frac12}_i,\ldots,\frac12).
\end{equation*}
\subsubsection{$\alpha=e_i-e_j$, $1<i<j$}
We have $\dprod{\lambda+\rho}{\alpha^\vee}=j-i$.
So we must consider $m$ with 
\begin{equation}\label{m1}
0<mp<j-i.
\end{equation}
Consider
\begin{equation*}
(\nu_s)_{s=1}^\ell:=\lambda+\rho-mp\alpha=(\ldots,\underbrace{\ell-i+1/2-mp}_i,\ldots)
\end{equation*}
Since
\begin{equation*}
\ell-i+\frac12>\ell-i+\frac12-mp>\ell-i+\frac{1}{2}-(j-i)=\ell-j+\frac12,
\end{equation*}
It follows that $\nu_i=\nu_s$ for some $s$ with $i<s<j$, so Lemma~\ref{critstar} applies.
The total contribution to (\ref{jsum}) from positive roots of this form is 
zero.
\subsubsection{$\alpha=e_i+e_j$, $1<i<j$}
We have $\dprod{\lambda+\rho}{\alpha^\vee}=2\ell-i-j+1$.
So we must consider $m$ with 
\begin{equation}\label{m2}
0<mp<2\ell-i-j+1.
\end{equation}
Consider
\begin{equation*}
(\nu_s)_{s=1}^\ell:=\lambda+\rho-mp\alpha=(\ldots,\underbrace{\ell-i+1/2-mp}_i,\ldots,
\underbrace{\ell-j+1/2-mp}_j,\ldots).
\end{equation*}
If $\ell-i+1/2-mp<0$ then 
\begin{equation*}
\abs{\ell-i+1/2-mp}=mp-(\ell-i+\frac12)< (2\ell-i-j+1)-(\ell-i+\frac12)
=\ell-j+\frac12,
\end{equation*}
so $\abs{\nu_i}=\abs{\nu_s}$ for some $s>j$, and Lemma~\ref{critstar} applies.
So assume $\ell-i+\frac12-mp>0$. Then $\nu_i=\nu_s$ for some
$s>i$, unless $\ell-i+\frac12-mp=\ell-j+\frac12$, i.e $mp=j-i$.
Then 
\begin{equation*}
\abs{\nu_j}=\abs{\ell-j+\frac12-mp}<\abs{\ell-j+\frac12}+(j-i)
=\ell-i+1,
\end{equation*}
so $\abs{\nu_j}=\nu_s$ for some $s\leq i$ and Lemma~\ref{critstar} applies.
The total contribution to (\ref{jsum}) from positive roots of this form is
zero.
\subsubsection{$\alpha=e_i$, $1<i$}
We have $\dprod{\lambda+\rho}{\alpha^\vee}=2\ell-2i+1$.
So we must consider $m$ with 
\begin{equation}\label{m3}
0<mp<2\ell-2i+1.
\end{equation}
Consider
\begin{equation*}
(\nu_s)_{s=1}^\ell:=\lambda+\rho-mp\alpha=
(\ldots,\underbrace{\ell-i+1/2-mp}_i,\ldots).
\end{equation*}
Since $\abs{\ell-i+\frac12-mp}<\ell-i+\frac12$, by (\ref{m3}),
we have $\abs{\nu_i}=\nu_j$ for some $j>i$, so Lemma~\ref{critstar} applies.
The total contribution to (\ref{jsum}) from positive roots of this form is
zero.

\subsubsection{$\alpha=e_1-e_j$, $j>1$}
We have $\dprod{\lambda+\rho}{\alpha^\vee}=r+j-1$.
So we must consider $m$ with 
\begin{equation}\label{m4}
0<mp<r+j-1.
\end{equation}
Consider
\begin{equation*}
(\nu_s)_{s=1}^\ell:=\lambda+\rho-mp\alpha=
(r+\ell-\frac12-mp,\ell-\frac32,\ldots,\underbrace{\ell-j+\frac12+mp}_j,\ldots).
\end{equation*}
By (\ref{m4}) 
\begin{equation*}
\ell-\frac12>r+\ell-\frac12-mp>r+\ell-\frac12-r-j+1=\ell-j+\frac12.
\end{equation*}
Thus $\nu_1=\nu_s$ for some $s$ with $1<s<j$ and Lemma~\ref{critstar} applies.
The total contribution to (\ref{jsum}) from positive roots of this form is 
zero.

\subsubsection{$\alpha=e_1$}
We have $\dprod{\lambda+\rho}{\alpha^\vee}=2r+2\ell-1$.
So we must consider $m$ with 
\begin{equation}\label{m5}
0<mp<2r+2\ell-1.
\end{equation}
Consider
\begin{equation*}
(\nu_s)_{s=1}^\ell:=\lambda+\rho-mp\alpha=
(r+\ell-\frac12-mp,\ell-\frac32,\ldots).
\end{equation*}
Then $\nu_1$ will have the same size as some $\nu_i$, $i>1$ if
$r+\ell-\frac12-mp>0$, and also if
$r+\ell-\frac12-mp<0$ but $mp-r-\ell+\frac12< \ell-\frac12$.
So we may assume $mp-r-\ell+\frac12 \geq \ell-\frac12$, hence
$mp\geq r+2\ell-1$.
Thus, the only $m$ we have to consider is the unique $m$
(if it exists) which satisfies
\begin{equation}\label{dm1}
r+2\ell-1\leq mp<2r+2\ell-1.
\end{equation}
Assuming $m$ exists, the conjugate of $\lambda+\rho-mp\alpha$
in $X_+$ is 
\begin{equation*}
(mp-r-\ell+\frac12,\ell-\frac32,\ldots,\frac12),
\end{equation*}
obtained from $\lambda+\rho-mp\alpha$ by changing the sign of $\nu_1$,
so has sign $-1$.
Subtracting $\rho$ gives
\begin{equation*}
(mp-r-2\ell+1,0,\ldots,0)=(mp-2\ell+1-r)\omega_1.
\end{equation*}
The total contribution to (\ref{jsum}) from $e_1$ is 
therefore zero unless there exists $m$ satisfying (\ref{dm1}),
in which case $m$ is unique and the contribution to
(\ref{jsum}) is equal to $v_p(mp)\chi(r_1\omega_1)$,
where $r_1=mp-r-2\ell+1$.

\subsubsection{$\alpha=e_1+e_j$, $j>1$}
We have $\dprod{\lambda+\rho}{\alpha^\vee}=r+2\ell-j$.
So we must consider $m$ with 
\begin{equation}\label{m6}
0<mp<r+2\ell-j.
\end{equation}
Consider
\begin{equation*}
(\nu_s)_{s=1}^\ell:=\lambda+\rho-mp\alpha=
(r+\ell-\frac12-mp,\ell-\frac32,\ldots,\underbrace{\ell-j+\frac12-mp}_j,\ldots).
\end{equation*}
If $r+\ell-\frac12-mp<0$ then by (\ref{m6})
\begin{equation*}
\abs{r+\ell-\frac12-mp}=mp-(r+\ell-\frac12)<\ell-j+\frac12,
\end{equation*}
so $\abs{\nu_1}=\nu_s$ for some $s>j$.
So we may assume $r+\ell-\frac12-mp > 0$. Then $\nu_1$ equals
some $\nu_s$ for $s>1$ unless $r+\ell-\frac12-mp=\ell-j+\frac12$, i.e.
\begin{equation}\label{jm1}
j=mp-r+1.
\end{equation}
Next we consider $\nu_j$. Clearly if $\nu_j>0$, then $\nu_j=\nu_s$
for some $s>j$, so we can assume $\nu_j<0$.
Next if $\abs{\nu_j}=mp-(\ell-j+\frac12)\leq\ell-\frac32$, we
have $\abs{\nu_j}=\nu_s$ for some $s$, by the pigeon-hole principle.
Thus, we may also assume that $mp-(\ell-j+\frac12)\geq\ell-\frac12$,
i.e $mp\geq 2\ell-j$. We have shown that by Lemma~\ref{critstar}, a pair
$(m,j)$ contributes nothing to (\ref{jsum}), except when
\begin{equation}\label{jm2}
2\ell-j\leq mp<r+2\ell-j,  \quad \text{and $j=mp-r+1$}.
\end{equation}
Thus, if $m$ exists, it is the unique integer satisfying
\begin{equation}\label{mp}
2\ell-1+r\leq 2mp<2r+2\ell-1.
\end{equation}
 For this $(m,j)$, we have 
\begin{equation*}
\lambda+\rho-mp\alpha=(\ell-j+\frac12,\ldots,\underbrace{\ell-2mp+r-\frac12}_j,\ldots).
\end{equation*}
The Weyl group element mapping this element into $X_+$ is the transposition
of $1$ with $j$ followed by the sign change in the first coordinate. This
group element has sign 1. On subtracting $\rho$ from the result we end up with
\begin{equation*}
(2mp-2\ell+1-r, 0,\ldots,0)=(2mp-2\ell+1-r)\omega_1.
\end{equation*}

The total contribution to (\ref{jsum}) from positive roots of this form is 
zero unless there exists $m$ satisfying (\ref{mp}) in which case
the contribution is $-v_p(mp)\chi(r_1\omega_1)$, where
$r_1=2mp-2\ell+1-r$.

\subsubsection{}\label{summaryB}
By considering the above cases, we see that the sum (\ref{jsum}) is zero
unless there exists an integer $M$ with 
\begin{equation*}
r+2\ell-1\leq Mp<2r+2\ell-1.
\end{equation*}
In the latter case, let $r_1=Mp-2\ell+1-r$.
Then there is a contribution of $v_p(Mp)\chi(r_1\omega_1)$
from the root $e_1$ and when $M$ is even there is a contribution
of $-v_p(\frac{M}{2}p)\chi(r_1\omega_1)$ from the root $e_1+e_j$,
where $j=\frac{M}{2}p-r+1$. This means that if $M$ exists 
the value of (\ref{jsum})  for $\lambda=r\omega_1$ is:
\begin{equation*}
\sum_{i>0}\Ch V(\lambda)^i=\begin{cases} v_p(Mp)\chi(r_1\omega_1),\quad\text{if $M$ is odd.}\\
\chi(r_1\omega_1),\quad\text{if $M$ is even and  $p=2$.}\\
0,\quad\text{if $M$ is even and $p$ is odd.}
\end{cases}
\end{equation*}

\subsection{Weyl module structures for type $B_\ell$}\label{submoduleB}
The above computations show that 
$\chi(r\omega_1)$ is  equal to either  $\Ch L(r\omega_1)$
or to $\Ch L(r\omega_1)+e\chi(r_1\omega_1)$, for a certain $r_1$
with $0\leq r_1<r$ 
and some positive integer $e$. 
We note that if we repeat the process setting $\lambda=r_1\omega_1$,
and denote the parameters (corresponding to $m$ and $r_1$) 
arising  from any nonsimple case by $m_1$ and $r_2$, then we see
that $r_2\equiv r\pmod p$. Since we assume $r\leq p-1$, this is impossible.
This shows that $V(r_1\omega_1)$ is irreducible in all cases
where $V(r\omega_1)$ is not.
Thus, $\rad V(r\omega_1)$ is either zero
or else all of its composition factors are isomorphic to a simple
Weyl module $V(r_1\omega_1)$.
Since there are no self-extensions of simple modules (\cite[II.2.12]{Ja})
we see that $\rad V(r\omega_1)$ is either zero or a direct
sum of $e$ copies of the simple Weyl module $V(r_1\omega_1)$

\begin{lemma}\label{multfree} 
$V(r\omega_1)$ is either simple or else $e=1$.
\end{lemma}
\begin{proof}
We first treat the special case $p=2$. Then $r=1$ and
$V(\omega_1)\cong V$. In this case, it is well known
that $V$ has a one-dimensional $G$-fixed subspace $T$ with $V/T$
irreducible. Thus the lemma holds in this case. 
%changed a little
Now asssume $p>2$. Then $V\cong V^*$ as $G$-modules, whence all tensor powers
of $V^*$ are also self-dual. Since $r<p$, $S^r(V^*)$ is 
naturally a direct summand of the $r$-th tensor power of $V^*$,
so $S^r(V^*)$ is also self-dual.
For $0\leq r\leq p-1$, $S^r(V^*)$ has a good filtration \cite[4.1(4)]{AJ},
and the factors are the modules $H^0((r-2j)\omega_1)$, for
$0\leq j\leq\lfloor\frac{r}{2}\rfloor$, each with multiplicity one, as
can be seen from the character equation 
\begin{equation*}
\Ch(S^r(V^*))-\Ch(S^{r-2}(V^*))=\chi(r\omega_1),
\end{equation*}
which follows from the definition of $H^0(r\omega_1)$.
Therefore,
\begin{equation*}
\begin{split}
\dim \Hom_G(V(r_1\omega_1), V(r\omega_1))&
=\dim \Hom_G(H^0(r\omega_1),H^0(r_1\omega_1))\\
&\leq\dim \Hom_G(S^r(V^*), H^0(r_1\omega_1))\\
&=\dim\Hom_G(V(r_1\omega_1), S^r(V^*))\\
&\text{(by self-duality of $S^r(V^*)$)}\\
&=\text{multiplicity of $H^0(r_1\omega_1)$}\\
&\text{in a good filtration of $S^r(V^*)$}\\
&\leq 1.
\end{split}
\end{equation*}
\end{proof}

The proof of Theorem~\ref{ThmB} can now be completed. Parts (a),
(b), and (c) have been proved in \S\ref{summaryB} and \S\ref{submoduleB}
and part (d) is obtained by substitution of Weyl module dimensions,
given in \S\ref{Weyldimensions}.

%}}}

%{{{ D computations

\section{Computations for $D_\ell$}
\subsection{Evaluation of the Sum Formula}
We fix $r$ with $0\leq r\leq p-1$ and set $\lambda=r\omega_1=(r,0,\ldots,0)$.
Here, we determine the contribution of each positive root to the right hand side
of (\ref{jsum}).
We have $\rho=(\ell-1,\ell-2,\ldots, 1, 0)$, so
\begin{equation*} 
\lambda+\rho=(r+\ell-1,\ell-2,\ldots,\underbrace{\ell-i}_i,\ldots,\ldots,1,0).
\end{equation*}
\subsubsection{$\alpha=e_i-e_j$, $1<i<j$}
We have $\dprod{\lambda+\rho}{\alpha^\vee}=j-i$.
So we must consider $m$ with 
\begin{equation}\label{m7}
0<mp<j-i.
\end{equation}
Consider
\begin{equation*}
(\nu_s)_{s=1}^\ell:=\lambda+\rho-mp\alpha
=(\ldots,\underbrace{\ell-i-mp}_i,\ldots,\underbrace{\ell-j+mp}_j,\ldots).
\end{equation*}
We have 
\begin{equation*}
\ell-i>\nu_i=\ell-i-mp>\ell-i-(j-i)=\ell-j,
\end{equation*}
so $\nu_i=\nu_s$ for some $s$ with $i<s<j$, and Lemma~\ref{critstar} applies.
The total contribution to (\ref{jsum}) from positive roots of this form is
zero.
\subsubsection{$\alpha=e_i+e_j$, $1<i<j$}
We have $\dprod{\lambda+\rho}{\alpha^\vee}=2\ell-i-j$.
So we must consider $m$ with 
\begin{equation}\label{m8}
0<mp<2\ell-i-j.
\end{equation}
Consider
\begin{equation*}
(\nu_s)_{s=1}^\ell:=\lambda+\rho-mp\alpha=(\ldots,
\underbrace{\ell-i-mp}_i,\ldots,\underbrace{\ell-j-mp}_j,\ldots).
\end{equation*}
If $\ell-i-mp\leq 0$, then by (\ref{m8})
\begin{equation*}
\abs{\ell-i-mp}=mp-\ell+i<2\ell-i-j-\ell+i=\ell-j,
\end{equation*}
so $\abs{\nu_i}=\nu_s$ for some $s>j$ and Lemma~\ref{critstar} applies.
So we can assume that $\nu_i=\ell-i-mp>0$. Then  $\nu_i$ equals
to $\nu_s$ for $s>i$ unless $\ell-i-mp=\ell-j$, i.e. $mp=(j-i)$.
In that case we have
\begin{equation*}
\abs{\nu_j}=\abs{\ell-j-mp}=\abs{(\ell-j)-(j-i)}<(\ell-j)+(j-i)=\ell-i,
\end{equation*}
so $\abs{\nu_j}=\nu_s$ for some $s>i$, and Lemma~\ref{critstar} applies.
The total contribution to (\ref{jsum}) from positive roots of this form is 
zero.

\subsubsection{$\alpha=e_1-e_j$, $j>1$}
We have $\dprod{\lambda+\rho}{\alpha^\vee}=r+j-1$.
So we must consider $m$ with 
\begin{equation}\label{m9}
0<mp<r+j-1.
\end{equation}
Consider
\begin{equation*}
(\nu_s)_{s=1}^\ell:=\lambda+\rho-mp\alpha=
(r+\ell-1-mp,\ell-2,\ldots,\underbrace{\ell-j+mp}_j,\ldots).
\end{equation*}
Then by (\ref{m9})
\begin{equation*}
\ell-2\geq r+\ell-1-mp>\ell-j,
\end{equation*}
so $\nu_1=\nu_s$ for some $s$ with $1<s<j$, and Lemma~\ref{critstar} applies.
The total contribution to (\ref{jsum}) from positive roots of this form is 
zero.

\subsubsection{$\alpha=e_1+e_j$, $j>1$}
We have $\dprod{\lambda+\rho}{\alpha^\vee}= r+2\ell-j-1$.
So we must consider $m$ with 
\begin{equation}\label{m10}
0<mp<r+2\ell-j-1.
\end{equation}
Consider
\begin{equation*}
(\nu_s)_{s=1}^\ell:=\lambda+\rho-mp\alpha=
(r+\ell-1-mp,\ell-2,\ldots,\underbrace{\ell-j-mp}_j,\ldots).
\end{equation*}
If $\nu_1=r+\ell-1-mp\leq0$, then by (\ref{m10})
\begin{equation*}
\abs{\nu_1}=mp-r-\ell+1<(r+2\ell-j-1)-r-\ell+1=\ell-j,
\end{equation*}
so $\abs{\nu_1}=\nu_s$ for some $s>j$. We can therefore assume
that $\nu_1=r+\ell-1-mp>0$. Then we have $\nu_1=\nu_s$ for some $s>1$
unless $r+\ell-1-mp=\ell-j$, i.e 
\begin{equation}\label{jmp1}
j=mp-r+1.
\end{equation}
Next we consider $\nu_j$. Clearly if $\nu_j>0$, then $\nu_j=\nu_s$
for some $s>j$, so we can assume $\nu_j<0$.

Suppose $mp\leq 2\ell-j-2$. Then
\begin{equation*}
\abs{\nu_j}=mp-\ell-j\leq 2\ell-j-2-\ell+j=\ell-2,
\end{equation*}
so $\abs{\nu_j}=\nu_s$ for some $s$, by the pigeon-hole principle.
Thus, we may also assume that $mp\geq 2\ell-j-1$.
 We have shown that by Lemma~\ref{critstar}, a pair $(m,j)$ contributes nothing to 
(\ref{jsum}), except when
\begin{equation}\label{jmp2}
 2\ell-j-1\leq mp<r+2\ell-j-1 \quad\text{and $j=mp-r+1$}.
\end{equation}
Thus, if $m$ exists, it is the unique integer satisfying
\begin{equation}\label{jmp3}
r+2\ell-2\leq 2mp<2r+2\ell-2.
\end{equation}
 For this $(m,j)$, we have 
\begin{equation*}
\lambda+\rho-mp\alpha=(\ell-j,\ell-2\ldots,\underbrace{r+\ell-2mp-1}_j,\ldots,0).
\end{equation*}
The Weyl group element mapping this element into $X_+$ is the transposition
of $1$ with $j$ followed by the sign changes in the first coordinate
and the uniquely determined $s$-th coordinate with $\nu_s=0$. ($s$
is either $\ell$ or $1$). This group element has sign -1. 
On subtracting $\rho$ from the result we end up with
\begin{equation*}
(2mp-2\ell+2-r, 0,\ldots,0)=(2mp-2\ell+2-r)\omega_1.
\end{equation*}

The total contribution to (\ref{jsum}) from positive roots of this form is 
zero unless there exists $m$ satisfying (\ref{jmp3}) in which case
the contribution is $v_p(m)\chi(r_1\omega_1)$,
where $r_1=2mp-2\ell+2-r$.

\subsubsection{}\label{summaryD}
By considering the above cases, we see that the sum (\ref{jsum}) is zero
unless there exists an even integer $M$ with 
\begin{equation*}
r+2\ell-2\leq Mp<2r+2\ell-2.
\end{equation*}
In the latter case, let $r_1=Mp-2\ell+2-r$. Then  there is a contribution of 
of $-v_p(\frac{M}{2}p)\chi(r_1\omega_1)$ from the root $e_1+e_j$,
where $j=\frac{M}{2}p-r+1$. This means that if $M$ exists 
the value of (\ref{jsum})  for $\lambda=r\omega_1$ is:
\begin{equation*}
\sum_{i>0}\Ch V(\lambda)^i=\begin{cases} v_p(\frac{M}{2}p)\chi(r_1\omega_1),\quad\text{if $M$ is even.}\\
0,\quad\text{if $M$ is odd.}
\end{cases}
\end{equation*}

\subsection{Weyl Module structure for type $D_\ell$}\label{submoduleD}
The above computations show that 
$\chi(r\omega_1)$ is  equal to either  $\Ch L(r\omega_1)$
or to $\Ch L(r\omega_1)+e\chi(r_1\omega_1)$, for a certain $r_1$
with $0\leq r_1<r$ 
and some positive integer $e$. 
We note that if we repeat the process setting $\lambda=r_1\omega_1$,
and denote the parameters (corresponding to $m$ and $r_1$) 
arising  from any nonsimple case by $m_1$ and $r_2$, then we see
that $r_2\equiv r\pmod p$. Since we assume $r\leq p-1$, this is impossible.
This shows that $V(r_1\omega_1)$ is irreducible in all cases
where $V(r\omega_1)$ is not.
Thus, $\rad V(r\omega_1)$ is either zero
or else all of its composition factors are isomorphic to a simple
Weyl module $V(r_1\omega_1)$.
Since there are no self-extensions of simple modules (\cite[II.2.12]{Ja})
we see that $\rad V(r\omega_1)$ is either zero or a direct
sum of $e$ copies of the simple Weyl module $V(r_1\omega_1)$.

%This ios the same lemma as in the type B section. Rewrite it for type D only
\begin{lemma}\label{multifree} 
$V(r\omega_1)$ is either simple or else $e=1$.
\end{lemma}
\begin{proof}
We have $V\cong V^*$ as $G$-modules, and we can argue exactly
as for the case $p>2$ of Lemma~\ref{multfree}.
\end{proof}

The proof of Theorem~\ref{ThmD} can now be completed. Parts (a)
and (b) have been proved in \S\ref{summaryD} and \S\ref{submoduleD}
and part (c) is obtained by substitution of Weyl module dimensions,
given in \S\ref{Weyldimensions}.

%}}}

%{{{ A computations

\section{Computations for $A_\ell$, $\lambda= r(\omega_1+\omega_{\ell})$}\label{sectionA}
Let $\lambda= r(\omega_1+\omega_{\ell})$, $0\leq r\leq p-1$.
\subsection{Evaluation of the Sum Formula}\label{subsectionA}
Fix $0\leq r\leq p-1$ and set 
$\lambda = r(\omega_1 + \omega_{\ell}) = r(1, 0,\ldots,0, -1)$. 
Here, we determine the contribution of each positive root to the right hand side
of (\ref{jsum}).
Recall that $\rho=(\frac{\ell}{2},\frac{\ell}{2}-1,  \ldots,  -\frac{\ell}{2}+1,  -\frac{\ell}{2})$
Hence $\ds \lambda + \rho= ( \frac{\ell}{2}+r, \frac{\ell}{2}-1,  \frac{\ell}{2}-2,
\ldots,-\frac{\ell}{2}+1, -\frac{\ell}{2}-r)$.
% \subsection{\textit{Case 1} $\alpha=e_i-e_j$, $1<i<j<n$}
\subsubsection{$\alpha=e_i-e_j$,  $1<i<j<\ell+1$}
We have $\dt = j-i$. So we must consider $m$ with

\begin{equation*}
0< mp < \dt = j-i.
\end{equation*}
Consider
\begin{equation*}
(\nu_s)_{s=1}^\ell:=\n=( \ldots, \underbrace{\frac{\ell}{2}-i+1-mp}_i , \ldots, \underbrace{\frac{\ell}{2}-j+1+mp}_j , \ldots).
\end{equation*}
Then
\begin{equation*}
i<mp+i<j
\end{equation*}
and we have $\nu_{mp+i}=\frac{\ell}{2}-(mp+i)+1=\nu_i$. 
By Lemma~\ref{critstar},
the contribution to the sum (\ref{jsum}) from positive roots of this type is zero.
\subsubsection{$\alpha = e_1-e_i$ or $e_i-e_{\ell+1}$, $1< i < \ell+1$}
We have $\dt = r+i-1$, so we must consider $m$ satisfying
\begin{equation*}
0< mp < r+i-1.
\end{equation*}
Consider
\begin{equation*}
(\nu_s)_{s=1}^\ell:=\n=(\frac{\ell}{2}+r -mp , \ldots, \underbrace{\frac{\ell}{2}-i+1+mp}_i , \ldots, -\frac{\ell}{2}-r).
\end{equation*}
Since $r\leq p-1$ we also have  have $mp>r$. So 
\begin{equation*}
1<mp-r+1<i.
\end{equation*}
Then, $\nu_{mp-r+1}= \frac{\ell}{2} - mp +r=\nu_1$.
By Lemma~\ref{critstar}, the contribution to the sum from this type of positive roots is zero.
By symmetry, the contribution from the roots of the form $\alpha=e_i-e_{\ell+1}$ is also zero.
\subsubsection{$\alpha=e_1-e_{\ell+1}$}\label{onlynonzero}
We have $\dt=\ell+2r$, so we must consider $m$ with 
\begin{equation*}
0< mp < 2r+\ell.
\end{equation*}
Consider
\begin{equation*}
(\nu_s)_{s=1}^\ell:=\n = (\frac{\ell}{2}+r -mp , \frac{\ell}{2}-1 ,   \ldots, 1 -\frac{\ell}{2},  -\frac{\ell}{2}-r+mp).
\end{equation*}

Suppose $mp\leq r+\ell-1$. Then $r+\frac{\ell}{2}-mp\geq 1-\frac{\ell}{2}$, so
$\nu_1=\nu_s$ for some $s$ with $2\leq s\leq\ell$ and Lemma~\ref{critstar}
applies. Thus, in order for all $\nu_i$ to be distinct
we must have 
\begin{equation}\label{Aineq}
r+\ell-1<mp<2r+\ell.
\end{equation}
Since $(2r+\ell)-(r+\ell-1)=r+1\leq p$, there is at most one $m$
satisfying (\ref{Aineq}).
If such an $m$ exists, then the transposition $(1,\ell+1)$
takes $\lambda-mp\alpha$ to 
\begin{equation*}
(mp-\frac{\ell}{2}-r,\frac{\ell}{2}-1,\ldots,1-\frac{\ell}{2},-mp+\frac{\ell}{2}+r)\in X_+
\end{equation*}
and when we subtract $\rho$, the resulting weight is $(mp-\ell-r)(\omega_1+\omega_\ell)$.
\subsection{Contributions combined}\label{summaryA}
By considering the above cases, we see that the sum (\ref{jsum})
is zero unless there exists an integer $m$ satisfying the condition (\ref{Aineq})
and that if $m$ exists, then the value of (\ref{jsum}) for $\lambda=r(\omega_1+\omega_\ell)$ is
\begin{equation*}
\sum_{i>0}\Ch V(\lambda)^i= v_p(mp)\chi(r_1(\omega_1+\omega_\ell)),
\end{equation*}
where $r_1=mp-\ell-r$.
We claim that $V(r_1(\omega_1+\omega_\ell))$ is simple. If not, then the above
analysis applied to $r_1$ would show that $\rad V(r_1(\omega_1+\omega_\ell))$
would have the character of $V(r_2(\omega_1+\omega_\ell))$, where
$0\leq r_2=m'p-\ell-r_1$, for some integer $m'$. But then $r_2\equiv r\pmod p$,
contradicting the hypothesis that $0\leq r\leq p-1$. This proves our claim.
\subsection{Structure of $V(\lambda)$}\label{submoduleA}
We have shown either $V(\lambda)$ is  simple or else 
its radical has character equal to some positive integer $e$ times the character of  
the simple Weyl module $V(r_1(\omega_1+\omega_\ell))$.
We now show that in the latter case, we must have $e$=1.

\begin{lemma} Let $G$ be of type $A_\ell$, $\ell\geq 3$ and let $0\leq r\leq p-1$.
Then $H^0(r(\omega_1+\omega_\ell))$ is either simple or else
$H^0(r(\omega_1+\omega_\ell))/L(r(\omega_1+\omega_\ell))$ is simple
and isomorphic to $H^0(r_1(\omega_1+\omega_\ell))$, where $r_1=mp-\ell-r$.
\end{lemma}
\begin{proof}
The following character formula is classical and
a special case of the Littlewood-Richardson rule:
\begin{equation}\label{classicalprod}
\chi(r\omega_1)\chi(r\omega_\ell)=\sum_{s=0}^r\chi(s(\omega_1+\omega_\ell)).
\end{equation}
Then, since $H^0(r_1(\omega_1+\omega_\ell))$ is simple and self-dual, while
$H^0(r\omega_1)$ and $H^0(r\omega_\ell)$ are simple and dual to each other, we have
\begin{equation*}
\begin{aligned}
1\leq e&=\dim\Hom_G(H^0(r(\omega_1+\omega_\ell)), L(r_1(\omega_1+\omega_\ell)) )\\
&\leq \dim\Hom_G(H^0(r\omega_1)\otimes H^0(r\omega_\ell), L(r_1(\omega_1+\omega_\ell)) )\\
&=\dim\Hom_G(L(r_1(\omega_1+\omega_\ell),H^0(r\omega_1)\otimes H^0(r\omega_\ell)) )\\
&=\dim\Hom_G(V(r_1(\omega_1+\omega_\ell),H^0(r\omega_1)\otimes H^0(r\omega_\ell)) )\\
&\leq 1,
\end{aligned}
\end{equation*}
by (\ref{classicalprod}) and the multiplicity property of good filtrations.
\end{proof}

The proof of Theorem~\ref{ThmA} can now be completed. Parts (a)
and (b) have been proved in \S\ref{summaryA} and \S\ref{submoduleA}
and part (c) is obtained by substitution of Weyl module dimensions,
given in \S\ref{Weyldimensions}.

%}}}

%{{{ Case A'

\section{Computations for type $A_\ell$, $\lambda=\omega_2+\omega_{\ell-1}$}\label{sectionA'}
Let  $\lambda=\omega_2+\omega_{\ell-1}$ , $\ell\geq 4$.
\subsection{Evaluation of the Sum Formula}\label{subsectionA'}
Here, we determine the contribution of each positive root to the right hand side
of (\ref{jsum}).
We have $\lambda = (1, 1,0,\ldots, 0, -1, -1)$, $\rho=(\el, \el-1, \el-2, \ldots, -\el+2, -\el+1, -\el)$, so $\lambda + \rho = (\el+1, \el, \el-2, \ldots, \el-i+1, \ldots, -\el+2, -\el, -\el-1)$.
We search for $m$'s satisfying
\begin{equation} \label {1}
 0<mp<\langle \lambda + \rho, \alpha \rangle.
 \end{equation}
Since $p\geq2$ and $m\neq0$ the inequality (\ref{1}) is equivalent to
\begin{equation} \label{mm2}
2 \leq  mp < \langle \lambda+\rho, \alpha \rangle.
\end{equation}

We find $m$'s and $\alpha$'s satisfying (\ref{mm2}). Then we find out whether $\lambda+\rho-mp\alpha$ has any two coordinates of the same value. 

\subsubsection{$\alpha = e_1-e_2$ or $\alpha = e_{\ell}-e_{\ell+1}$}
This case is trivial since  $\dt = 1$.
Therefore for any $p$ and $\ell$, the  contribution to the sum from the positive roots of this type is zero.

\subsubsection{$\alpha = e_1-e_i$ or  $e_i-e_{\ell+1}$, $2<i<\ell$}
For this case $\dt = i$, hence
\begin{equation}\label{m2c2}
 2 \leq mp <i.
 \end{equation}
We consider
\begin{equation*}
(\nu_s)_{s=1}^\ell:=\n=(\el+1-mp, \el,  \el-2, \ldots, \underbrace{\el-i+1 + mp,}_i \ldots).
\end{equation*}
By (\ref{m2c2}) we have
\begin{equation*}
\el+1-i < \el+1-mp \leq \el -1.
\end{equation*}
Thus unless $\el +1 - mp = \el -1$, there is an $s$ with $3\leq s <i$ such that $\nu_1=\nu_s$.
If $\el +1 - mp = \el -1$,  then  $mp=2. $ Hence the equality holds if and only if $p=2$ and $m=1$. In this case,
\begin{equation*}
(\nu_s)_{s=1}^\ell:=\n=(\el -1, \el , \el-2, \ldots, \underbrace{\el-i+3}_i, \ldots, -\el+2, -\el, -\el-1).
\end{equation*}
Since $i>2$ we get 
\begin{equation*}
\el+2-i < \el+3-i < \el +1.
\end{equation*}
Hence there exists $s$ with $1 \leq s \leq i-1$ such that $\nu_i = \nu_s$ in this case also. 
Therefore for any $p$ and $\ell$, the  contribution to the sum from the positive roots of this type is zero.
By symmetry we see that contribution to the sum from the roots of type 
$\alpha = e_i-e_{\ell+1}$, $2<i<\ell$ is also zero. 

\subsubsection{$\alpha = e_1-e_{\ell}$ or $e_2-e_{\ell+1}$}
For this case $\dt = \ell+1$. Hence
\begin{equation}\label{m2c3}
2 \leq mp < \ell+1.
\end{equation}
We consider
\begin{equation*}
(\nu_s)_{s=1}^\ell:=\n= (\el+1-mp, \el, \el-2, \ldots, -\el+2, -\el+mp, -\el-1).
\end{equation*}
From (\ref{m2c3}) we get
\begin{equation*}
-\el+2 \leq -\el+mp < \el +1.
\end{equation*}
The only way that $\nu_{\ell}$ is different from the other coordinates of $\nu$ is that $\nu_{\ell}=-\el+mp=\el-1$ which gives us $ mp=\ell-1$. In this case,
\begin{equation*}
\nu_1 = \el +1 -mp = -\el+2 = \nu_{\ell-1}.
\end{equation*}
Hence there is at least two equal coordinates of $\nu$.
Therefore for any $p$ and $\ell$, the  contribution to the sum from the positive roots of this type is zero.
By symmetry, the contribution to the sum from the root $\alpha = e_2-e_{\ell+1}$ is also zero. 

\subsubsection{ $\alpha = e_2-e_i$ or $e_i-e_{\ell}$, $2<i<\ell$}
For this case $\dt = i-1$. Hence
\begin{equation}\label{m2c4}
2 \leq mp < i-1.
\end{equation}
We consider
\begin{equation*}
(\nu_s)_{s=1}^\ell:=\n=(\el+1,  \el - mp , \el-2, \ldots, \el-i+1+mp, \ldots ).
\end{equation*}
From (\ref{m2c4}) we get
\begin{equation*}
\el - i +1 < \el - mp \leq \el -2
\end{equation*}
Thus there is an $s$ with $3\leq s < i$ such that $\nu_2=\nu_s$.
Therefore for any $p$ and $\ell$, the  contribution to the sum from the positive roots of this type is zero.
By symmetry, the contribution to the sum from the roots of type $\alpha = e_i-e_{\ell}$, $2<i<\ell$ is also zero. 

\subsubsection{$\alpha = e_1-e_{\ell +1}$}
For this case $\dt = \ell + 2 $. Hence,
\begin{equation}\label{m2c5}
2 \leq mp < \ell+2.
\end{equation}
We consider 
\begin{equation*}
(\nu_s)_{s=1}^\ell:=\n= (\el+1 -mp, \el, \el-2, \ldots, -\el+2,  -\el,  -\el-1+mp ).
\end{equation*}
With (\ref{m2c5}) we get
\begin{equation*}
-\el-1<\el+1-mp\leq \el-1.
\end{equation*}
Thus the $\nu_i$ are distinct if and only if one of the following conditions hold.
\begin{enumerate}
\item[(i)] $\nu_1 =\el-1$.
\item[(ii)] $\nu_1 = -\el +1$.
\end{enumerate}
The condition (i) holds if and only if $mp=2$. Thus unless $p=2$ and $m=1$, 
$\nu_1$ is equal to one of the other  $\nu_i$.
If $p=2$ and $m=1$, then
\begin{equation*}
(\nu_s)_{s=1}^\ell:=\n = ( \el-1,  \el,  \el-2,  \ldots  , -\el+2,   -\el ,   -\el+1).
\end{equation*}
When we conjugate this element into $X_+$ with the permutation 
$(1,2)(\ell,\ell+1)$ and subtract $\rho$ we obtain
$(0, 0, \ldots,0)$. The permutation has positive sign.

The condition (ii) holds if and only if $ mp = \ell$. In this case,
\begin{equation*}
(\nu_s)_{s=1}^\ell:=\n=(-\el+1, \el, \el-2, \ldots , -\el+2, -\el, \el-1).
\end{equation*}
When we conjugate this element into $X_+$ with the permutation 
$(1,\ell-1,\ell,2)$ and subtract $\rho$ we obtain
$(0, 0, \ldots,0)$. The permutation has negative sign.

\subsubsection{$\alpha = e_2-e_{\ell}$}
For this case $\dt = \ell  $. Hence
\begin{equation}\label{m2c6}
2 \leq mp < \ell .
\end{equation}
We consider
\begin{equation*}
(\nu_s)_{s=1}^\ell:=\n=(\el+1, \el-mp, \el-2, \ldots, -\el+2,  -\el+mp,  -\el-1 ).
\end{equation*}
Hence $\lambda-mp\alpha$ has distinct coordinates if and only if $\nu_2 = -\el+1$. That is $mp +1 = \ell$.
In this case,
\begin{equation*}
\lambda-mp\alpha = (\el+1, -\el+1,  \el-2,  \ldots, -\el+2, \el-1, -\el-1).
\end{equation*}
So its Weyl conjugate in $X_+$ is
\begin{equation*}
( \el+1, \el-1, \el-2, \ldots, -\el+2, -\el+1, -\el-1)
\end{equation*}
and on subtraction of $\rho$ we obtain
$\omega_1+\omega_{\ell}$.
The Weyl group  element has negative sign.

\subsection{Contributions combined}\label{summaryA'}
If $p=2$ and $\ell$ is even, we have
\begin{equation}\label{2elleven}
\sum_{i>0}\Ch V(\lambda)^i= (v_2(\ell)-1)\chi(0).
\end{equation}
Thus, $\Ch(\rad V(\lambda))=e_0\chi(0)$, where $e_0\geq 0$. 
For $p=2$ and odd $\ell$,
\begin{equation}\label{2ellodd}
\sum_{i>0}\Ch V(\lambda)^i= -\chi(0)+v_2(\ell-1)\chi(\omega_1+\omega_{\ell}).
\end{equation}
Thus, $\Ch(\rad V(\lambda))=-\chi(0)+e_1\chi(\omega_1+\omega_\ell)$, where $e_1\geq 1$.

If $p$ is odd $V(\lambda)$ is simple unless $\ell = 0\pmod p$  or $\ell = 1\pmod p$. If $ \ell = 0\pmod p$ then 
\begin{equation}\label{pdivell}
\sum_{i>0}\Ch V(\lambda)^i= v_p(\ell)\chi(0).
\end{equation}
Thus, $\Ch(\rad V(\lambda))=e_2\chi(0)$, where $e_2\geq 1$.

If $\ell = 1\pmod p$ then 
\begin{equation}\label{1modp}
\sum_{i>0}\Ch V(\lambda)^i= v_p(\ell-1)\chi(\omega_1+\omega_{\ell}).
\end{equation}
Thus, $\Ch(\rad V(\lambda))=e_3\chi(\omega_1+\omega_\ell)$, where $e_3\geq 1$.

It is a well known special case of Theorem~\ref{ThmA} that
\begin{equation}\label{Vom1ell}
\chi(\omega_1+\omega_{\ell})=\begin{cases} \Ch L(\omega_1+\omega_{\ell})+\chi(0)\quad
\text{if $\ell\equiv -1\pmod p$,}\\
\Ch L(\omega_1+\omega_{\ell})\quad\text{otherwise.}
\end{cases}
\end{equation}
Thus, when $\ell = 1\pmod p$, we see that
$V(\omega_1+\omega_\ell)$ is simple if $p$ is odd and has 
radical isomorphic to $k$ if $p=2$.

\subsection{Structure of $V(\lambda)$}\label{submoduleA'}
In this subsection we shall prove Theorem~\ref{ThmA'}. We begin with some general
facts which do not depend on $p$.
We observe that $V(\omega_2)$ is simple and isomorphic to $\wedge^2(V)$,
where $V$ is the standard $\ell+1$-dimensional module, and that $V(\omega_{\ell-1})\cong \wedge^{\ell-1}(V)$ is isomorphic to the dual of $V(\omega_2)$. 
Thus, 
\begin{equation}\label{homkM}
\Hom_G(k, L(\omega_2)\otimes L(\omega_{\ell-1}))\cong \Hom_G(L(\omega_2),L(\omega_2))\cong k.
\end{equation}
A fundamental property of Weyl modules and induced modules
(\cite[II.4.19]{Ja}) is that the tensor product of two Weyl
modules or two induced modules has a filtration by modules of the same kind,
with the factors being the same as for the corresponding
tensor product in characteristic zero.
Thus, the classical identity \cite[p.225]{FH}:
\begin{equation}\label{LR}
\chi(\omega_2)\chi(\omega_{\ell-1})=\chi(\omega_2+\omega_{\ell-1})+\chi(\omega_1+\omega_\ell)+\chi(0).
\end{equation}
implies that $M:=V(\omega_2)\otimes V(\omega_{\ell-1})$
has a filtation $E_1\subset E_2\subset M$ with factors
$E_1\cong V(\omega_2+\omega_{\ell-1})$, $E_2/E_1\cong V(\omega_1+\omega_\ell)$
and $M/E_2\cong k$.
It follows that 
\begin{multline}\label{homkV}
\dim\Hom_G(k, V(\omega_2+\omega_{\ell-1}))\leq \dim\Hom_G(k, M)\\
=\dim\Hom_G(L(\omega_2),L(\omega_2))=1.
\end{multline}
Dually, we know that $M$ has a quotient 
isomorphic to $H^0(\omega_2+\omega_{\ell-1})$.
Since the highest weight $\omega_2+\omega_{\ell-1}$
occurs in $M/\rad E_1$, the latter also has $H^0(\omega_2+\omega_{\ell-1})$
as a quotient.
Therefore, if we write $[A:L(\mu)]$ for the multiplicity of the simple $G$-module
$L(\mu)$ as a composition factor of the $G$-module $A$, we have
\begin{equation}\label{compmult}
\begin{split}
[V(\omega_2+\omega_{\ell-1}):L(\omega_1+\omega_\ell)]&=
[H^0(\omega_2+\omega_{\ell-1}):L(\omega_1+\omega_\ell)]\\
&\leq [M/\rad E_1:L(\omega_1+\omega_\ell)]\\ 
&=[M/E_1:L(\omega_1+\omega_\ell)]\\  
&=[E_2/E_1:L(\omega_1+\omega_\ell)]+[M/E_2:L(\omega_1+\omega_\ell)]\\
&=1.\end{split}
\end{equation}
Now we are ready to prove each part of Theorem~\ref{ThmA'}.
Parts (c) and (f) are immediate, since by \S\ref{summaryA'} the sum
formula evaluates to zero.
In (a) and (d), we see from (\ref{2elleven}) and (\ref{pdivell}) that 
$\rad V(\lambda)$ has only trivial composition factors, with positive
multiplicity. Since $\Ext^1_G(k,k)=0$, it follows from (\ref{homkV})
that the multiplicity is one.
In (b) we know from (\ref{1modp}) and (\ref{Vom1ell}) that the only
composition factor of $\rad V(\lambda)$, with positive multiplicity, is
$L(\omega_1+\omega_\ell)\cong V(\omega_1+\omega_\ell)$. By (\ref{compmult})
the multiplicity is one.
In (g) we see from (\ref{2ellodd}) and (\ref{Vom1ell}) that the only
composition factor of $\rad V(\lambda)$, with positive multiplicity, is
$L(\omega_1+\omega_\ell)$. By (\ref{compmult}) the multiplicity is one.
The proof of (e) is more involved.
Since $\dim L(\omega_2)=\binom{\ell+1}{2}$ is odd, 
$M\cong\End(L(\omega_2))$ decomposes as the direct sum of
the scalar multiples of the identity map and the space of trace zero endomorphisms. 
The indecomposable submodule $E_1$ must therefore be 
isomorphic to a submodule of the nontrivial summand.
Then, by (\ref{homkM}), it follows that
\begin{equation}\label{nok}
\Hom_G(k,V(\lambda))=0.
\end{equation}
The Sum Formula tells us that the composition factors (with positive
multiplicities) of $\rad V(\lambda)$
are $k$ and $L(\omega_1+\omega_\ell)$. By (\ref{compmult}), we know that
$[V(\lambda):L(\omega_1+\omega_\ell)]=1$. 
Then by (\ref{nok}), it follows that $\soc(\rad V(\lambda))\cong L(\omega_1+\omega_\ell)$.
Since $\Ext^1_G(k,k)=0$, the module  $(\rad V(\lambda))/\soc(\rad V(\lambda))$
must be a trivial module of dimension $\leq \dim\Ext^1_G(k, L(\omega_1+\omega_\ell))$.
By (\ref{Vom1ell}), 
 the self-duality of $L(\omega_1+\omega_\ell)$  and a
standard property of Weyl modules \cite[II.2.14]{Ja},
 we know that
\begin{equation}\label{ext}
\Ext^1_G(k,L(\omega_1+\omega_\ell))\cong \Ext^1_G(L(\omega_1+\omega_\ell),k)\cong k.
\end{equation}
Hence $\rad V(\lambda)$ is a nonsplit extension
of $k$ by $L(\omega_1+\omega_\ell)$. Such a module is unique up to
isomorphism, by (\ref{ext}), and isomorphic to $H^0(\omega_1+\omega_\ell)$,
by (\ref{Vom1ell}).
This completes the proof of (e).

Finally, (h) is obtained by substitution of Weyl module dimensions.
given in \S\ref{Weyldimensions} and Theorem~\ref{ThmA}.
This completes the proof of Theorem~\ref{ThmA'}.

%}}}

%{{{ Case A''

\section{Further computations for Type $A_4$}
The aim of this section is to prove Theorem~\ref{ThmA''}.
Let $\lambda=(p-1)(\omega_2 + \omega_3)$ and  $\lambda_1=(p-2)(\omega_2+\omega_3)$.
First we observe that for $p=2$, we have $\lambda=\omega_2+\omega_3$, so the structure of $V(\lambda)$ has been given in Theorem~\ref{ThmA'}. Its radical is isomorphic to $k$.
So (a) is proved, as also is (e). We assume that $p$ is odd from now on.
We begin with (c). Let $\lambda_2=(p-2)(\omega_1+\omega_4)$.
Then $V(\lambda_2)$ is simple, by Theorem~\ref{ThmA}.
Next, we compute the Sum Formula (\ref{jsum}) for
$\lambda_1=(p-2,p-2,0,2-p,2-p)$. We have $\rho = (2, 1, 0, -1, -2)$. 
Using Lemma~\ref{critstar} we see that the only nonzero contribution is from
the positive root $(0, 1, 0, -1, 0)$ and for $m=1$. This yields
\begin{equation*}
\sum_{i>0}\Ch V(\lambda_1)^i=\chi(\lambda_2).
\end{equation*}
%Since simple modules do not extend themselves, by \cite[II.2.12]{Ja},
%we see that $\rad V(\lambda_1)$ is a direct sum.
It follows that $\rad V(\lambda_1)\cong V(\lambda_2)$, proving (c).
To prove (b) we compute (\ref{jsum}) for $\lambda=(p-1, p-1, 0, 1-p, 1-p)$.
Direct calculations show that for all $\alpha$, except 
$\alpha = (1, 0, 0,  0, -1)$, either there is no $m$ satisfying $0<mp<\dt$, or $\bir$ has repeated entries. When  $\alpha = (1, 0, 0, 0, -1)$, $m$ can be $1$ or $2$.
Consider the case $\alpha = (1, 0, 0, 0, -1)$, and $m=1$. Then
$\bir = (1, p,  0,  -p,  -1) $, 
so the permutation sending $\bir$ to its conjugate $(p,  1, 0, -1,-p)\in X_+$ has positive sign.  
Subtraction of $\rho$ from this conjugate yields
$(p-2) (\omega_1+\omega_4)$. 
Hence $\kay = \chi((p-2)(\omega_1+\omega_4))$.
In the case $m=2$, we get $\bir = (-p+1, p, 0, -p, p-1)$ which is sent to its conjugate
$(p, p-1, 0, 1-p, -p)\in X_+$ by a permutation with negative sign.
Subtraction of $\rho$ from this conjugate yields $(p-2)(\omega_2 + \omega_3)$. 
Hence $\kay = - \chi((p-2)(\omega_2+\omega_3))$.
For both cases $v_p(mp)=1$, since $p>2$.
Thus, 
\begin{equation}\label{tally}
\sum_{i>0}\Ch V(\lambda)^i=\chi(\lambda_1)-\chi(\lambda_2)=\Ch L(\lambda_1),
\end{equation}
where the last equality is by (c).
It follows that $\rad V(\lambda)\cong L(\lambda_1)$ and (b) is proved.

By (b) and (c) we have
\begin{equation*}
\dim L(\lambda)= \dim V(\lambda)-\dim V(\lambda_1)+\dim V(\lambda_2).
\end{equation*}
Then the proof of (d) is completed by substituting the dimensions of Weyl modules
from \S\ref{Weyldimensions}.

%$ \lambda_1 = (p-2)(\omega_2 + \omega_3)$\\
%$ \lambda_2 = (p-2)(\omega_1 + \omega_4)$\\

% The following is the proof that\\

% $\ds dim( L(\lambda))= dim( V(\lambda))- dim( V(\lambda_1)) + dim( V(\lambda_2)) $

%The idea is as follows:

%First show that :\\

%1. $ V(\lambda_2)$ is irreducible. So we know its dimension.\\

%2. $ V(\lambda_1)= L(\lambda_1)+V(\lambda_2)$\\

%..........

%We have $\ds Ch(V(\lambda)) = \Ch(L(\lambda))+ \sum_{x=1}$

%So we track $m$'s $\alpha$'s and the entries of $\lambda+\rho-mp\alpha$.\\

%Let $r=p-1$, then $\lambda = ( p-1,p-1,0,-p+1,-p+1)$, and $\rho = (2,1,0,-1,-2)$ hence $\lambda + \rho =(p+1, p, 0, -p, -p-1)$

%$\alpha$ ,
%$\langle \lambda+\rho, \alpha \rangle$ ,
%$m$ ,
%$\lambda+\rho-mp\alpha$ ,
%$\{\lambda+\rho-mp\alpha\}$ ,
%$sign(\omega)$ ,
%$char(V)$ , contribution

%}}}

%{{{ finite geometry and groups 

\section{$p$-ranks of incidence matrices }\label{prankproofs}

In this section, we explain how the theorems on $p$-ranks
stated in \S\ref{prankresults} may be deduced from our results on
the dimensions of simple modules. We shall use some of the notation
from that earlier discussion. 

\subsection{The permutation module on singular points}
We recall that $P$ is either the set of isotropic $1$-dimensional subspaces
in the $n$-dimensional space $V(q)$ with respect to a nondegenerate
quadratic form $f$, or else the set of singular $1$-dimensional subspaces
in $V(q^2)$ with respect to a nonsingular Hermitian form $h$.
In both cases, the set of polar hyperplanes is denoted by $P^*$.
The map sending each $p\in P$ to its polar hyperplane $p^\perp\in P^*$
is clearly a permutation isomorphism with respect to the action of
the projective orthogonal group $\PO(V(q),f)$ in the quadratic case
or  the projective unitary group $\PGU(V(q^2),h)$ in the Hermitian case.
For technical reasons, rather than working with these groups, 
it will be more convenient for us to work with certain subgroups which
are {\it Chevalley groups}, under which term we shall include
 the twisted as well as the untwisted types. 
Let $\Om(V(q),f)$ denote the commutator subgroup of the orthogonal group
$\Orth(V(q),f)$ and let $\POm(V(q),f)$ be its image in $\PO(V(q),f)$.
By \cite[Theorems 11.3.2 and 14.5.2]{Ca}, $\POm(V(q),f)$ is, with one exception,
a Chevalley group of type $B_\ell(q)$, $D_\ell(q)$ or $^2D_\ell(q)$, with 
$\dim V(q)=2\ell+1$ in the first case and $2\ell$ in the other two cases. 
The exceptional case is when $q=2$ and $\dim V(q)=5$, in which case the
Chevalley group of type $B_2(2)$ is a subgroup of $\PO(V(q),f)$
containing $\POm(V(q),f)$ with index 2.
By \cite[Theorems 14.5.1]{Ca}, $\PSU(V(q^2),h)$ is a (twisted) Chevalley group
of type $^2A_\ell(q)$, where $\dim V(q^2)=\ell+1$.
In each case, we shall denote the Chevalley group by ${\overline G(q)}$.

The permutation module over $k$ with basis $P$
can be identified with the space $k[P]$ of all functions from $P$ to $k$
by identifying a point with its characteristic function.
Being a permutation module, it is a self-dual
%%corrected to G(q) from \Orth(q)
$k{\overline G(q)}$-module, with $P$ as an orthonormal basis of a nonsingular 
${\overline G(q)}$-invariant symmetric bilinear form.
Since $\abs{P}\equiv 1\pmod p$,
we have an orthogonal direct sum decomposition of $k{\overline G(q)}$-modules
\begin{equation}\label{1plusY}
k[P]=k\allone \oplus Y_P,
\end{equation}
where $\allone$ is the constant function with value 1 and
$Y_P$ is the subspace of functions $f$ such that $\sum_{p\in P}f(p)=0$.
Let $\pi_\allone$ and $\pi_{Y_P}$ be the projections with respect to this decomposition.
Thus, $\pi_1$ maps each $p\in P$ to $\allone$.
The {\it head} of $Y_P$, $\head(Y_P)$, is defined as the largest 
semisimple quotient of $Y_P$ and the {\it socle}, $\soc(Y_P)$, is
the maximal semisimple submodule. Let $L=\soc(Y_P)$. The self duality
of $k[P]$ and (\ref{1plusY}) imply that $Y_P$ is self-dual
and hence that $\head(Y_P)\cong L^*$.
%added as kG(q)-modules
%fixed proof of claim below
%fix imprecision that V(q) is not an \overline G module
We claim that $L^*\cong L$ as $k{\overline G(q)}$-modules.
If ${\overline G(q)}$ is of type $B_\ell(q)$ then every simple module is self-dual.
This is also true for type $D_\ell(q)$ or $^2D_\ell(q)$  when $\ell$ is even.
Suppose ${\overline G(q)}\cong \POm(V(q),f)$ 
is of type $D_\ell(q)$ or $^2D_\ell(q)$ with $\ell$ odd.
Then there  is an automorphism $\delta$ of $\Om(V(q),f)$,
such that twisting a simple module by $\delta$ yields the dual simple module.  
If $M$ is the module of a representation of $\Om(V(q),f)$,
we denote by $M^\delta$ the twisted module obtained by composing
the representation with $\delta$. Then $V(q)\cong V(q)^\delta$ and since
$k[P]$ is constructed from $V(q)$, it follows that $k[P]\cong k[P]^\delta$.
Hence $Y_P\cong (Y_P)^\delta$ and $L\cong L^\delta\cong L^*$.
Similarly, in the Hermitian case, the group $\SU(V(q^2),h)$
has an automorphism $\delta$ which dualizes simple modules.
Therefore, the form $h$ gives an isomorphism of $V(q^2)^\delta$ with the 
Galois twist $V(q^2)^{(q)}$. Since $k[P]$ is constructed from $V(q^2)$, 
it follows that $k[P]^{(q)}\cong k[P]^\delta$.
Then, because $k[P]$ is isomorphic to all of its Galois conjugates,
we have $k[P]\cong k[P]^\delta$, from which we get $L\cong L^\delta\cong L^*$.
This completes the proof of our claim.
\begin{lemma}\label{endodims}\ 
%add assumptions on $\ell$?
\begin{enumerate}
\item[(a)]The action of ${\overline G(q)}$ on $P$
has permutation rank $3$.
\item[(b)]$\dim\End_{k{\overline G(q)}}(k[P])=3$.
\item[(c)]$\dim\End_{k{\overline G(q)}}(Y_P)=2$.
\end{enumerate}
\end{lemma}
\begin{proof}

Part (a) is well known \cite{D}. (We should point out that our 
hypotheses on $\ell$ ($\ell\geq 3$ for type $A_\ell$, $\ell\geq 2$ for type $B_\ell$, and $\ell\geq 3$ for type $D_\ell$ are needed here.)
Then (b) follows from the general fact that 
the endomorphism ring of a transitive permutation module
has a natural basis in bijection with the set of double cosets
of a stabilizer. Therefore, the cardinality of the basis 
equals the rank of the permutation representation.
Part (c) now follows from the decomposition (\ref{1plusY}).
\end{proof}

By means of the isomorphism of ${\overline G(q)}$-sets between $P^*$ and $P$ the
incidence matrix between $P^*$ and $P$ can be interpreted as the map
\begin{equation*}
\phi: k[P]\to k[P], p\mapsto\sum_{p'\in p^\perp} p',
\end{equation*}
which clearly belongs to $\End_{k{\overline G(q)}}(k[P])$.

Let ${\overline G(q)}_P$ be the stabilizer of an element of $P$. Thus, 
the possible choices for ${\overline G(q)}_P$ form a conjugacy class in ${\overline G(q)}$. Later we shall 
make  a definite choice.

\begin{lemma}\ 
\begin{enumerate}
\item[(a)] $k[P]$ is not semisimple.
\item[(b)] $L$ is simple.
\item[(c)] $L$ is, up to isomorphism, the unique nontrivial simple $k{\overline G(q)}$-module
on which ${\overline G(q)}_P$ has nonzero fixed points. 
\item[(d)] $\image\phi=k\allone\oplus L$.
\end{enumerate}
\end{lemma}
\begin{proof}
A $3$-dimensional semisimple $k$-algebra is a direct product
of $3$ copies of $k$ so has no nilpotent elements.
We see that $\phi^2(p)=\sum_{p''\in P}a_{p,p''}p'$, with
\begin{equation*}
a_{p,p''}=\abs{\{p'\mid p'\in p^\perp, p''\in p'^\perp\}}=\abs{p^\perp\cap p''^\perp}\pmod p=\allone.
\end{equation*}
Similarly, we see that $\phi\circ\pi_{\allone}=pi_{\allone}=\pi_{\allone}\circ\phi$.
Thus $(\phi-\pi_{\allone})^2=0$, which proves (a).
To prove (b), we fix an isomorphism from $\head(Y_P)$
with $L=\soc(Y_P)$ and let $\psi\in\End_{k{\overline G(q)}}(Y_P)$ be the composite map
$Y_P\to \head(Y_P)\cong L\subset Y_P$.
Given a decomposition of $L$ into a direct sum of $m$ simple modules,
 the $m$ projections onto these simple summands  are linearly independent
elements of $\End_{k{\overline G(q)}}(L)$. Their compositions with $\psi$,
together with $\id_{Y_P}$, give $m+1$ linearly independent elements of 
$\End_{k{\overline G(q)}}(Y_P)$. Hence $m=1$,  by Lemma~\ref{endodims}(c).
Part (c) follows from (b) and  Frobenius reciprocity, since $k[P]$ can be viewed as the
induced module $\ind_{{\overline G(q)}_P}^{{\overline G(q)}}(k)$.
To prove (d), observe that  $\pi_{\allone}$, $\pi_{Y_P}$ and 
$\psi\circ\pi_{Y_P}$ form a basis of
$\End_{k{\overline G(q)}}(k[P])$. It follows the image of a nonzero
$k{\overline G(q)}$-endomorphisms of $k[P]$ must be one of
$\allone$, $Y_P$, $L$, $k[P]$ and $k\allone\oplus L$.
Since $\phi\neq\pi_{\allone}$ and  $\phi^2=\pi_{\allone}$, it follows 
$\image\phi=k\allone\oplus L$.
\end{proof}

\begin{corollary}\label{reduction1} $\rank_pA_{11}=1+\dim L$, where $L$ is the unique nontrivial
simple $k{\overline G(q)}$-module on which ${\overline G(q)}_P$ has nontrivial fixed points.
\end{corollary}

\subsubsection{The dual Hermitian quadrangle}
Let $(V(q^2), h)$ be a 5-dimensional Hermitian space and now  define
$P$ to be is the set of $2$-dimensional totally singular subspaces,
instead of the one-dimensional ones considered above.
This new $P$ is the set of points of the dual Hermitian quadrangle.
The above discussion carries through with 
${\overline G(q)}=\SU(V(q^2),h)$, which acts with rank 3 on $P$, and we obtain 
Corollary\ref{reduction1} in that case too.

\begin{remark}
In all of the cases above, the simplicity of the socle and head of $Y_P$
can also be derived from the very general theory in \cite{Cu}. Indeed that
theory shows that for the permutation
module on the cosets of any maximal parabolic subgroup in a finite group
with a split $(B,N)$-pair of characteristic $p$, there is a decomposition (\ref{1plusY})
such that the socle and head of the nontrivial summand are simple.
\end{remark}

\subsection{Representations of finite Chevalley groups}
The aim of this subsection is to obtain further information about the simple module $L$.
To each of the Chevalley groups ${\overline G(q)}$ in the previous subsection, there
corresponds a central extension $G(q)$ called the universal Chevalley group.
Any representation of ${\overline G(q)}$ gives a representation of $G(q)$ by composing
with the natural map $G(q)\to {\overline G}(q)$. If we let $G(q)_P$ be the full preimage 
in $G(q)$ of ${\overline G}(q)_P$, then $G(q)_P$ is the point stabilizer in
$G(q)$ and $L$ is the unique simple $G(q)$-module
on which $G(q)_P$ has nonzero fixed points. 
For technical reasons, it is preferable to work with $G(q)$. Since our goal is to determine $\dim L$, there is no loss in replacing ${\overline G}(q)$ by $G(q)$.

A basic fact \cite[13.1]{St2} about simple $kG(q)$-modules is that they are the restrictions
of certain simple rational modules for $G(k)$, the ambient Chevalley group over $k$.
By this we mean that $G(k)$ has an endomorphism $\sigma$, such that $G(q)$
is the subgroup of fixed points (\cite[12.4]{St2}). 
Further, we may assume that the subgroups $T$ and $B$ are $\sigma$-stable
and denote by $T(q)$ and $B(q)$ the subgroups of $T$ and $B$ fixed by $\sigma$.
$G(k)$ is a simply connected algebraic group, since $G(q)$ is universal.
In the notation \S\ref{alggps}, let $L(\lambda)$
denote the simple $G(k)$-module with highest weight $\lambda\in X_+$.
The simple module $L(\lambda)$ is characterized by the property
that it has a unique $B$-stable line, and  $T$ acts on this line
by the character $\lambda$.
By Steinberg's theorem \cite[13.3]{St2}, the simple $G(k)$ modules with
highest weights in the set 
\begin{equation*}
X_q=\{\lambda=\sum_{i=1}^\ell a_i\omega_i\in X_+ \mid \,  0\leq a_i\leq q-1\ (\forall i)\}
\end{equation*}
remain irreducible upon restriction to $G(q)$ and this gives a complete set
of mutually nonisomorphic simple $kG(q)$-modules.

 Our immediate goal is to identify the highest weight $\lambda\in X_q$ such that
the restriction of the $G(k)$-module $L(\lambda)$ to $G(q)$ is isomorphic
to $L$.

\begin{lemma}\ 
\begin{enumerate} 
\item[(a)] For $G(q)$ in Theorem~\ref{Oprank}  we have
$L=L((q-1)\omega_1)$.
\item[(b)] For $G(q)$ in Theorem~\ref{Aprank} we have $L=L((q-1)(\omega_1+\omega_\ell))$.
\item[(c)] For $G(q)$ in Theorem~\ref{dualHermitian} we have $L=L((q-1)(\omega_2+\omega_3))$.
\end{enumerate}
\end{lemma}
\begin{proof}
Consider the action of the $G(k)$ on its standard module $V$.
We may assume that $V$ is obtained from the standard module of $G(q)$
by extending scalars from $\F_q$ or $\F_{q^2}$ to $k$.
Then the highest weight space is isotropic. Let $G(k)_P$ be its stabilizer.
Because $T$ and $B$ are $\sigma$-stable, the highest weight space contains a vector $v$
in the standard module of $G(q)$.
The subgroup $G(q)_P$ was previously defined only up to conjugacy in $G(q)$.
Now we fix our choice by setting $G(q)_P=G(k)_P\cap G(q)$.
The weight of $v$ is $\omega_1$, so the element $v^{q-1}\in S^{q-1}(V)$ is
a weight vector of weight $(q-1)\omega_1$, this weight being the highest in
$S^{q-1}(V)$ and having multiplicity one.
It follows that $S^{q-1}(V)$ has a unique composition factor isomorphic
to $L((q-1)\omega_1)$ and that $G(k)_P$ stabilizes the highest weight space
in this simple module. Since the restriction of the weight $(q-1)\omega_1$
to $T(q)$ is trivial, it follows that $G(q)_P$ acts trivially on this
highest weight space. Therefore $L=L((q-1)\omega_1)$.

In the case of the unitary group 
the standard module $V$ of $G(k)$ has highest weight $\omega_1$ 
and the dual module $V^*$ has highest weight $\omega_\ell$.
Let $v$ and $\delta$ be, respectively, weight vectors with these weights. 
Then the vector $v^{q-1}\otimes \delta^{q-1}$ spans the highest weight
space in $S^{q-1}(V)\otimes S^{q-1}(V^*)$. Now there is a
$kG(q)$-isomorphism $V^{(q)}\cong V^*$, mapping the highest weight vector $v\in V^{(q)}$
to the highest weight vector $\delta\in V^*$.
Hence $G(q)_P$ stabilizes the span of $v^{q-1}\otimes \delta^{q-1}$, which
is the weight space of weight $(q-1)(\omega_1+\omega_\ell)$, the highest in
$S^{q-1}(V)\otimes S^{q-1}(V^*)$. It follows that $G(q)_P$
stabilizes the highest weight space in the simple module $L((q-1)(\omega_1+\omega_\ell))$
and since the restriction of $(q-1)(\omega_1+\omega_\ell)$ to $T(q)$
is equal to the restriction of $(q-1)(\omega_1+ q\omega_1)=(q^2-1)\omega_1$, which is trivial, we have shown that $G(q)_P$ fixes a nonzero vector in 
$L((q-1)(\omega_1+\omega_\ell))$, which implies that
$L\cong L((q-1)(\omega_1+\omega_\ell))$.

Finally, we give an analogous argument of which the case $\ell=4$  applies in 
the case (A'').
In this case we can find a totally singular $2$-dimensional 
subspace of the standard module $V(q^2)$ of $G(q)$ such that
$\wedge^2Z$ spans the highest weight space in $\wedge^2(V)=\wedge^2(V(q^2)\otimes_{\F_{q^2}}k)$. We define $G(q)_P$ to be the stabilizer in $G(q)$ of $Z$
and $G(k)_P$ to be the stabilizer of the highest weight
space (of weight $\omega_2$) in $\wedge^2(V)$.  Then 
$G(k)_P$ also stabilizes the highest weight space in in $S^{q-1}(\wedge^2(V))$.
Since $V^*\cong V^{(q)}$ as $kG(q)$-modules, it 
follows that $G(q)_P$ stabilizes the one-dimensional highest weight
(of weight $\omega_{\ell-1}$) in $\wedge^2(V^*)$ and thus also the
space of weight $(q-1)\omega_3$ in $S^{q-1}(\wedge^2(V^*))$ and finally the
highest weight space of weight $(q-1)(\omega_2+\omega_{\ell-1})$ in 
$S^{q-1}(\wedge^2(V)\otimes S^{q-1}(\wedge^2(V^*))$. This implies that
$G(q)_P$ stabilizes the highest weight space in 
$L((q-1)(\omega_2+\omega_{\ell-1}))$, and we see by inspection that
this space affords the trivial character of $T(q)$.
We conclude that $L\cong L((q-1)(\omega_2+\omega_{\ell-1}))$ in this case.
\end{proof}

Next we obtain a further refinement using Steinberg's Tensor Product Theorem
\cite[Theorem 1.1]{St}.

In the orthogonal cases the $p$-adic expression of the highest weight of $L$ is 
\begin{equation*}
(q-1)\omega_1=(p-1)\omega_1+p(p-1)\omega_1+\cdots+p^{t-1}(p-1)\omega_1.
\end{equation*}
So by Steinberg's Tensor Product Theorem, we have
\begin{equation*}
L=L((q-1)\omega_1)=L((p-1)\omega_1)\otimes
L((p-1)\omega_1)^{(p)}\otimes\cdots\otimes
L((p-1)\omega_1)^{(p^{t-1})}.
\end{equation*}

%rewrote a little
This is a twisted tensor product, with $t$ tensor factors, in which the
factor $L((p-1)\omega_1)^{(p^{i})}$ is the module
obtained from the simple $G(k)$-module of highest weight 
$(p-1)\omega_1$ by composing the representation with the
$i$-th power of the Frobenius endomorphism of $G(k)$.

In the unitary case we have the analogous result with $L((p-1)\omega_1)$ replaced by
$L((p-1)(\omega_1+\omega_\ell))$ and in the case of the dual Hermitian quadrangle we replace $L((p-1)\omega_1)$ 
with $L((p-1)(\omega_2+\omega_3))$.

We can summarize these findings in the following statement.

\begin{theorem}
\begin{enumerate}
\item[(a)] $\rank_p A_{11} = 1+(\dim_k L((p-1)\omega_1))^t$
in the orthogonal cases (Theorem~\ref{Oprank}).
\item[(b)] $\rank_p A_{11} = 1+(\dim_k L((p-1)(\omega_1+\omega_\ell)))^t$
in the Hermitian case (Theorem~\ref{Aprank}).
\item[(c)] $\rank_p A_{11} = 1+(\dim_k L((p-1)(\omega_2+\omega_3)))^t$
in case of the dual Hermitian generalized quadrangle (Theorem~\ref{dualHermitian}).
\end{enumerate}
\end{theorem}

Now the theorems of \S~\ref{prankresults}
follow from Theorems~\ref{ThmB}, \ref{ThmD}, \ref{ThmA}, and \ref{ThmA''}.

%}}}

\subsection*{Acknowlegdements}
We wish to thank Eric Moorhouse for helpful discussions and for
showing us his computer calculations. We are also  indebted to the
developers of the computer algebra system SAGE \cite{SAGE}, which
was used to work out many examples.

%{{{ biblio

%}}}

\end{document}